\documentclass[10pt,reqno]{amsart}
\usepackage{amsmath,amssymb,latexsym,soul,cite,mathrsfs}
\usepackage{color,enumitem,graphicx}
\definecolor{Orange}{RGB}{255,127,0}
\definecolor{Purple}{RGB}{255,127,0}
\usepackage[colorlinks=true,urlcolor=blue,
citecolor=blue,linkcolor=blue,linktocpage,pdfpagelabels,
bookmarksnumbered,bookmarksopen]{hyperref}
\usepackage[english]{babel}
\usepackage{enumitem}

\usepackage[left=1.6cm,right=1.6cm,top=1.7cm,bottom=1.7cm]{geometry}

\makeatletter
\newcommand{\leqnomode}{\tagsleft@true}
\newcommand{\reqnomode}{\tagsleft@false}
\makeatother

\date{}
\def\nd{\noindent}
\def\thend{\rule{3mm}{3mm}}

\newtheorem{theorem}{Theorem}[section]
\newtheorem{definition}{Definition}[section]

\newtheorem{proposition}{Proposition}[section]
\newtheorem{lemma}{Lemma}[section]
\newtheorem{remark}{Remark}[section]
\newtheorem{exm}{Example}[section]

\newcommand{\cqd}{\hspace{10pt}\fbox{}}

\newcommand{\eps}{\epsilon}

\newcommand{\R}{\mathbb{R}}

\thanks{Corresponding author: Claudiney Goulart}

\thanks{The first author was also partially supported by CNPq with grant 309026/2020-2. The second author was partially supported by Fapeg with grant 202110267000424}

\begin{document} 
	
	\title[Nonlocal elliptic problems with prescribed norm in the $L^p$-subcritical and $L^p$-critical growth
	]{Quasilinear nonlocal elliptic problems with prescribed norm in the $L^p$-subcritical and $L^p$-critical growth 
	}
	
	\vspace{1cm}

	\author{Edcarlos D. Silva}
	\address{Edcarlos D. Silva \newline  Universidade Federal de Goi\'as, IME, Goi\^ania-GO, Brazil}
	\email{\tt edcarlos@ufg.br}
	
	\author{J. L. A. Oliveira}
	\address{J. L. A. Oliveira \newline Universidade Federal de Goi\'as, IME, Goi\^ania-GO, Brazil }
	\email{\tt jefferson\_luis@discente.ufg.br}
	
	\author{C. Goulart}
	\address{C. Goulart \newline Universidade Federal de Jata\'i, Jata\'i-GO, Brazil }
	\email{\tt claudiney@ufj.edu.br}
	
	\subjclass[2010]{35J05,35J15,35J20,35J60} 
	\keywords{Normalized solution, Fractional $p$-Laplacian, Asymptotically periodic potentials, Periodic potentials, $L^p$-Subcritical growth, $L^p$-Critical growth}
	\begin{abstract}
		It is established existence of solution with prescribed $L^p$ norm for the following nonlocal elliptic problem:
		\begin{equation*}
			\left\{\begin{array}{cc}
				\displaystyle (-\Delta)^s_p u\ +\ V (x) |u|^{p-2}u\  = \lambda |u|^{p - 2}u + \beta\left|u\right|^{q-2}u\ \hbox{in}\ \mathbb{R}^N, \\
				\displaystyle \|u\|_p^p = m^p,\ u \in W^{s, p}(\mathbb{R}^N).
			\end{array}\right.
		\end{equation*}
		where $s \in (0,1), sp < N, \beta > 0 \text{ and } q \in (p, \overline{p}_s]$ where $\overline{p}_s =p+ sp^2/N$.
		The main feature here is to consider $L^p$-subcritical and $L^p$-critical cases. Furthermore, we work with a huge class of potentials $V$ taking into account periodic potentials, asymptotically periodic potentials, and coercive potentials. More precisely, we ensure the existence of a solution of the prescribed norm for the periodic and asymptotically periodic potential $V$ in the $L^p$-subcritical regime. Furthermore, for the $L^p$ critical case, our main problem admits also a solution with a prescribed norm for each $\beta > 0$ small enough.
	\end{abstract}

	\maketitle
	\section{Introduction}
	In this work, we investigate a class of problems that have attracted significant attention and been widely studied in recent years. More specifically, we consider the existence of normalized solution for the following nonlocal elliptic problem:
	\begin{equation}\label{P}\tag{$P_m$}
		\left\{\begin{array}{cc}
			\displaystyle (-\Delta)^s_p u\ +\ V (x) |u|^{p-2}u\  = \lambda |u|^{p - 2}u + \beta\left|u\right|^{q-2}u\ \hbox{in}\ \mathbb{R}^N, \\
			\displaystyle \|u\|_p^p = m^p,\ u \in W^{s, p}(\mathbb{R}^N),\end{array}\right.
	\end{equation}
	where the parameter $\lambda$ is determined by the Lagrange Multiplier Theorem, $N>ps, s\in(0,1), \beta > 0$ and $1 < p < q \leq \overline{p} = p + sp^2/N$. Later on, we shall consider the hypotheses on the potential $V: \mathbb{R}^N \to \mathbb{R}$ with $V \geq 0, V \not \equiv 0$.
	
	It is worthwhile to mention that the operator fractional $p$-Laplacian acts as follows:
	$$(- \Delta)^s_p u(x) := 2 \lim_{\varepsilon \to 0^+} \int_{\R^N \setminus B_\varepsilon(x)} \frac{|u(x) - u(y)|^{p- 2}(u(x) - u(y))}{|x - y|^{N+ sp}} dy $$
	with $p \in (1, \infty)$, $s \in (0, 1)$ and $N > sp$, see \cite{neza} for further details.
	
	It is crucial to note that nonlocal elliptic problems governed by the fractional p-Laplacian operator have been extensively studied in recent years, see  \cite{Vin2, tue, Maxfor, Lions, Lions2}. Recall that several contributions for this kind of operator have been considered in recent years under several different assumptions on the potential $V$ as well as on the nonlinear term $g(t) = \lambda |t|^{p - 2}t + \beta\left|t\right|^{q-2}t, t \in \mathbb{R}$. One of the main difficulties for this kind of problem is the lack of compactness from the fractional Sobolev spaces into the Lebesgue spaces. In order to overcome this difficulty some tools from the nonlinear analysis have been explored recovering some kind of compactness.  Furthermore, substantial research has been conducted on the fractional $p$-Laplacian and fractional Laplacian operators, particularly with regard to prescribed norm conditions. The pioneering work in \cite{J1} provides a foundational approach to nonlocal elliptic problems, and we refer readers to additional contributions such as \cite{Cin, Hirata, J2, j3, Clau2, Clau3, Clau5, al1, qua, frac1, wa} and the references therein. More recently, in \cite{calva} the authors considered the existence of normalized solutions for a problem with general potential $V: \mathbb{R}^N \to \mathbb{R}$ and nonlinearities. This study also addressed the case $V \leq 0$ and proved the existence of normalized solutions under further assumptions on  $V$.

For the semilinear local case, that is, putting $p = 2$ and $s = 1$, we observe that finding solutions with a prescribed norm is relevant in physical studies such as nonlinear optics and the theory of water waves due to the fact that the norm is preserved over time. Furthermore, the fractional Laplacian operator has been accepted as a model for diverse physical phenomena such as diffusion-reaction equations, quasi-geostrophic theory, Schr\"odinger equations, Porous medium problems, see for instance \cite{aka,biboa,ya,consta,pablo,las}. For the quasilinear case, we refer the reader to \cite{met,bjor} where some physical applications are discussed taking into account several applications such as continuum mechanics, phase transition phenomena, populations dynamics, image processes, game theory, see \cite{ber,cafa,las}.

 It is important to stress that 
quasilinear reaction-diffusion equations have attracted some attention in the last few years. The main motivation for this kind of problem is to combine nonlinear and quasilinear nonlocal terms to model a nonlinear diffusion. On this subject, we refer the reader to \cite{vasquez1,vasquez2,vasquez3} where many of these nonlinear nonlocal diffusion problems are considered.

	Our main contribution in the present work is to consider existence of solutions with a prescribed mass on $S_m$ where $S_m$ is the sphere in $L^p(\mathbb{R}^N)$. In fact, we prove the existence of solutions for the Problem \eqref{P} with $q \in (p, p + sp^2/N)$, assuming that
	the potential $V$ is periodic or asymptotically periodic, see Theorems \ref{TB1} and \ref{TB2} ahead. To do that, we apply a hypothesis introduced in \cite{Elves} which has been used in several works, see for example \cite{Marc1,elv3}. On the other hand, inspired in part by \cite{Claud4}, we also show the existence of critical points for the energy functional $J_\eps$ given by \eqref{jeps}, see also Theorem \ref{TB3} ahead. Furthermore, we also consider the $L^p$-critical case, that is, $q= p + sp^2/N$ proving the existence of a solution for the Problem \eqref{P} where $\beta>0$ is small enough, see Theorem \ref{TB4} ahead. Under our assumptions we emphasize that the potential $V$ can be zero in some subsets $\Omega \subset \mathbb{R}^N$, that is, we assume that $V(x) \geq 0, x \in \Omega$. As a consequence, we consider potentials $V$ where the embedding from the fractional Sobolev spaces into the Lebesgue spaces is not standard. In fact, taking into account that $V$ can be zero in some subsets $\Omega \subset \mathbb{R}^N$, we consider a vast class of potentials $V$. For this kind of problem we refer the interested reader to seminal works \cite{bart1,Rab, Weth, mi}.
	
	It is important to mention that there exist several works considering the nonlocal elliptic problems driven by the fractional $p$-Laplacian or for the fractional Laplacian involving normalized solutions, see \cite{plap3,plap1, loo}. In the work \cite{plap3}, the authors established the existence of solutions for a problem involving the p-Laplacian operator with a potential $V \geq 0$. In the case $V\not \equiv 0$, they added the assumptions $$V \in L^{\infty}(\mathbb{R}^N),\ \inf_{x \in \mathbb{R}^N}V(x) = 0,\  \hbox{and}\ \lim_{|x| \to \infty } V(x) = \infty. $$
These conditions enabled the authors to demonstrate a compactness result, establishing the existence of solutions in both the$L^p$-subcritical and $L^p$-critical cases. We note that, in our work, the assumption $\displaystyle \lim_{|x| \to \infty } V(x) = \infty$ is not required, which presents distinct challenges in obtaining an existence result for solutions.

 In \cite{plap4} the authors investigated the following elliptic problem
	\begin{equation}\label{int3} - \Delta_p u + V(x)|u|^{p-2}u = \mu |u|^{r-2}u + |u|^{q-2}u\ \hbox{in}\ \mathbb{R}^N,     
	\end{equation}
	where $r = p$ or $r = 2$. Assuming that $V(x) = |x|^k,$ the authors proved that Problem \eqref{int3} has at least one nontrivial solution with the prescribed norm. Moreover, they considered some additional assumptions around the power $k$ and $p < q <  p^*$ with $p^* = Np/(N - p)$. In \cite{plap2} the authors studied the following $L^p$-critical problem
	\begin{equation}\label{int2} - \Delta_p u + V(x)|u|^{p-2}u = \mu |u|^{p-2}u + a|u|^{s-2}u    
	\end{equation}
	where $a \geq 0, p \in (1, N)$, $\mu \in \mathbb{R}$ and $s = p + p^2/N$. In that paper, the potential $V$ satisfies $\displaystyle \lim_{|x| \to \infty} V(x) = + \infty$ and $V(x) \geq 0, x \in \mathbb{R}^N$. Under these conditions, they showed that there exists $a^* > 0$ such that the Problem \eqref{int2} has at least one solution with prescribed norm for all $a \in [0, a^*)$. Furthermore, the authors considered also the following minimization problem: 
	$$e(a) = \inf \left\{E_a(u): u \in \mathcal{H}, \int_{\mathbb{R}^N}|u|^p dx= 1 \right\}$$ where 
	$$\mathcal{H} = \left\{u \in W^{1, p}(\mathbb{R}^N): \int_{\mathbb{R}^N}V(x)|u|^p dx < \infty\right\}.$$ In that work, they prove also that the energies of the solutions are strictly positive for each $a \in [0, a^*)$, that is, $e(a) > 0$. Moreover,  $\displaystyle \lim_{a \uparrow
		a^*}e(a) = e(a^*) = 0$.
	Regarding the fractional Laplacian operator several works involving norms fixed in $L^2$ have been developed in recent years. For this kind of problems we refer the reader to \cite{frac1,Cin,frac3}. 
	
	It is important to emphasize that the existence of normalized solutions associated with the fractional $p$-Laplacian operator have been considered in the last years, see for instance \cite{pfrac}. In that work, the authors established the existence of normalized solutions for Problem \eqref{P}, considering both the $L^p$-subcritical and $L^p$-supercritical cases assuming that $V \equiv 0$. The $L^p$-critical is also considered assuming that $V \leq 0$. The main ingredient in that work, among other things, is to employ the principle of concentration-compactness.  In the present work, our approach differs significantly, ensuring the existence of normalized solutions by applying specific hypotheses on the potential $V \geq 0$. Under our assumptions, we are able to demonstrate the existence of normalized solutions even when $V$ vanishes on some open subsets $\Omega \subset \mathbb{R}^N$ For further insights and additional results under different conditions on the potential $V$ we refer the reader to recent works, such as \cite{ch1, ch2, ch3}.

	\subsection{Assumptions and statement of the main results}
	In the present work our main objective is to find the existence of local minimizers $u \in S_m$ for the energy functional $J$ where $S_m$ is defined in \eqref{S}. Now, we introduce the set $\mathcal{F}$ which is inspired by \cite{Elves} as the class of functions $f \in C(\mathbb{R}^N) \cap L ^{\infty}(\mathbb{
		R}^N)$ such that for all $\varepsilon > 0$ the Lebesgue measure of the set $\{x \in \mathbb{R}^N:|f(x)| \geq \varepsilon\}$ is finite. Here we refer the reader to \cite{Marc1,elv3} and references therein.
	Now, we shall consider the following hypotheses:
	\begin{itemize}
		\item[($V_1$)] The potential $V \in L^{\infty}(\mathbb{R}^N)$ is $1-$ periodic and $V(x) \geq 0$, $V \not\equiv 0$ for all $x \in \mathbb{R}^N$.
		\item[($V_2$)] The potential $V  \in C(\mathbb{R}^N) \cap L ^{\infty}(\mathbb{
			R}^N)$ is asymptotically periodic, i.e., there exists a potential $V_{\theta}  \in C(\mathbb{R}^N) \cap L ^{\infty}(\mathbb{
			R}^N)$, $1$-periodic, $V_{\theta} \not\equiv 0$, with $ V_{\theta}(x) \geq V(x) \geq 0 $, $ V \not\equiv V_{\theta}$ and $V \not\equiv 0$ such that $V - V_{\theta} \in \mathcal{F}$.		
	\end{itemize}
At this stage, we consider our function space $X$ as follows 
	$$\displaystyle X = \left\{u \in W^{s,p}(\R^N): \int_{\R^N} V(x)|u|^p dx < \infty\right\}.$$
Notice that for the case where $V$ is a bounded potential we obtain that $X = W^{s,p}(\mathbb{R}^N)$. In general, the function space $X$ is only a closed subset of $W^{s,p}(\mathbb{R}^N)$. Now, we remember that $W^{s, p}(\mathbb{R}^N)$ is the fractional Sobolev space, see \cite{neza}. Here we consider some examples where the potential
$V$ vanishes on some nonempty subsets of $\mathbb{R}^N$. First, we observe that in view of Proposition \ref{reflexive} ahead, the functional space $X$ is a reflexive Banach space equipped with the norm: 
	\begin{equation}\label{norma}
		\|u\| = \left([u]^p + \int_{\R^N} V(x) |u|^p dx\right)^{\frac{1}{p}}, u \in X.
	\end{equation}
	It is well known that $X$ is a reflexive Banach space using the usual norm given by 
 \begin{equation}
     \|u\|_{*} = ([u]^p + \|u\|_p^p)^{1/´p}, u \in X.
 \end{equation}
 One of the main features in the present work is to prove that the norms 
 $\|\cdot\|$ and $\|\cdot\|_{*}$ are equivalents, see Proposition \ref{ob11} ahead. Hence, we can look for weak solutions for our main problem using the norm $\|\cdot\|$. 
 Recall that $[u]$ represents the well-known Gagliardo seminorm of function $u$ given by
	$$\displaystyle [u] = \left(\int_{\R^N}\int_{\R^N} \frac{|u(x) - u(y)|^p}{|x - y|^{N + sp}} dx dy\right)^{\frac{1}{p}}, u \in X.$$
	In the present work, the continuous embedding from the Sobolev spaces into the Lebesgue spaces does not work directly. In fact, by using the fact that $V$ can be zero in some subsets $\Omega \subset \mathbb{R}^N$, the standard continuous embedding is not verified. Hence, we need to apply a hypothesis introduced by Sirakov \cite{Sira}. Namely, we consider the following statement:
	\begin{itemize}\item[($V_3$)] $\displaystyle  \sigma = \displaystyle \inf_{u \in X} \left\{\|u\|^p : \|u\|_p^p = 1\right\} > 0.$
	\end{itemize}
The previous hypothesis is fundamental in order to recover the continuous embedding of $X$ into the Lebesgue spaces $L^t(\mathbb{R}^N)$ for each $t \in [p, p^*_s]$.
\begin{remark}  
		It is important to stress that our potential $V$ can be zero in some subsets of $\Omega \subset \mathbb{R}^N$. In other words, under our assumptions, we can consider some nonnegative potentials $V$ such that the set $[V = 0] = \{ x \in \mathbb{R}^N: V(x) = 0 \}$ is not empty. In fact, we shall give some examples of potentials $V$ potentials that vanish in some subset of positive Lebesgue measure.
	\end{remark}

 Now, we consider some examples of an asymptotically periodic potential that satisfies the condition $(V_2)$. Namely, we consider
	\begin{exm}
		Let $V_{\theta}$ be a $1$-periodic potential. Define $$V(x) = \allowdisplaybreaks \left(1 - \dfrac{1}{1 + |x|}\right)V_{\theta}(x),\ \hbox{with}\ x \in \mathbb{R}^N.$$  We claim that $\left(V - V_{\theta}\right) \in \mathcal{F}$. Indeed, consider $A_{\varepsilon} = \{x \in \mathbb{R}^N: |V(x) - V_{\theta}(x)| \geq \varepsilon\}$. We observe that $$V(x) - V_{\theta}(x) = -V_{\theta}(x) \left(\dfrac{1}{1 + |x|}\right).$$ Hence, using the Co-area formula  \cite{Evans}, we get 
		\begin{eqnarray}\nonumber \mu(A_{\varepsilon}) &=& \int_{A_{\varepsilon}} dx \leq \dfrac{1}{\varepsilon^{\ell}} \int_{A_{\varepsilon}} |V(x) - V_{\theta}|^{\ell} dx\leq \dfrac{C}{\varepsilon^{\ell}}\int_{\mathbb{R}^N}\dfrac{1}{\left(1+ |x|\right)^{\ell}} dx\\
			\nonumber \allowdisplaybreaks &=& K_{\varepsilon} \lim_{t \to \infty}\int_{0}^{t} \dfrac{r^{N-1}}{1 + r^{\ell}} dr = K_{\varepsilon} \int_{0}^{1} \dfrac{r^{N-1}}{1 + r^{\ell}} dr + K_{\varepsilon} \lim_{t \to \infty}\int_{1}^{t} \dfrac{r^{N-1}}{1 + r^{\ell}} dr\\
			\nonumber \allowdisplaybreaks &\leq& K_{\varepsilon} + K_{\varepsilon} \lim_{t \to \infty} \dfrac{t^{N - {\ell}}}{N - {\ell}} -  \dfrac{K_{\varepsilon}}{N - \ell} < \infty,\ \forall\ \ell > N.   
		\end{eqnarray}
		The desired claim is now proved.
	\end{exm}

 It is not hard to prove that the usual hypothesis 
 \begin{equation}\label{aleei}
     \displaystyle \lim_{| x| \to \infty} |V(x) - V_{\theta}(x)| = 0 
 \end{equation}
 implies that $(V - V_{\theta}) \in \mathcal{F}$. Moreover,  the last conditions is weak than the conditions given by \eqref{aleei}.
	More specifically, we shall consider an example such that the hypothesis $(V - V_{\theta}) \in \mathcal{F}$ is satisfied in such way that the well-known hypothesis \eqref{aleei} does not work anymore. Namely, we consider the following example:
	
	\begin{exm} Let us consider the potential  $$V(x) = \displaystyle 1 - \dfrac{1}{1 + |x|^{2\alpha} \sin^{2\rho}(|x|)}, x \in \mathbb{R}^N$$ and $V_{\theta} = 1$ with  $\alpha > (N + 2)/2, 0 < \rho < 1/2$.  Let us consider a sequence $(x_k) \subset \mathbb{R}^N$ such that  $|x_k| = k\pi,$ $k \in \mathbb{N}$. It follows that 
		$$|V(x_k) - V_{\theta}(x_k)| = \dfrac{1}{1 + (k \pi)^{2 \alpha} \sin^{2 \rho}(k \pi)} = 1.$$
		As a consequence, $\displaystyle \lim_{k \to \infty} |V(x_k) - V_{\theta}(x_k)| = 1 \not= 0$. On the other hand, we observe that $V - V_{\theta} \in \mathcal{F}$. Indeed, by using the Co-area Theorem, we obtain that 
		\begin{equation}
			\displaystyle \int_{\mathbb{R}^N} \dfrac{1}{1 + |x|^{2\alpha} \sin^{2\rho}(|x|)} dx =  K \int_{0}^{\pi - \delta} \dfrac{r^{N - 1}}{1 + r^{2 \alpha} \sin^{2 \rho}(r)} dr + K \int_{\pi - \delta}^{\infty} \dfrac{r^{N - 1}}{1 + r^{2 \alpha} \sin^{2 \rho}(r)} dr  
		\end{equation}
		In particular, we infer that $$\displaystyle \int_{0}^{\pi - \delta} \dfrac{r^{N - 1}}{1 + r^{2 \alpha} \sin^{2 \rho}(r)} dr < +\infty, \delta \in (0, 1).$$ Furthermore,  over the interval $(\pi - \delta, \infty)$, we consider sub-intervals of the form $I_k = [k \pi - \delta, k \pi + \delta]$, $J_k = (k \pi + \delta, (k+1) \pi - \delta)$ where $k \in \mathbb{N}$ and $0 < \delta < 1$. In this way, by using the assumptions on $\alpha$ and $\rho$, we infer also that
		\begin{equation*} \displaystyle
			\displaystyle   \int_{0}^{\infty} \dfrac{r^{N - 1}}{1 + r^{2 \alpha} \sin^{2 \rho}(r)} dr = \left(\int_{\bigcup_{k = 1}^{\infty} I_k} + \int_{\bigcup_{k = 1}^{\infty} J_k}\right) \dfrac{r^{N - 1}}{1 + r^{2 \alpha} \sin^{2 \rho}(r)} dr + \int_{0}^{\pi - \delta}\dfrac{r^{N - 1}}{1 + r^{2 \alpha} \sin^{2 \rho}(r)} dr < \infty.
		\end{equation*}
		Hence, $V - V_{\theta} \in L^1(\mathbb{R}^N)$ is satisfied. Now, we consider again the set $A_{\varepsilon} = \{x \in \mathbb{R}^N: |V(x) - V_{\theta}(x)| \geq \varepsilon \}$. We claim that $\mu(A_{\varepsilon}) < \infty$. In fact, we observe that
		\begin{eqnarray}\nonumber \displaystyle \mu(A_\varepsilon) & = & \int_{A_{\varepsilon}} dx = \frac{1}{\varepsilon} \int_{A_{\varepsilon}} \varepsilon dx \le \frac{1}{\varepsilon}\int_{A_{\varepsilon}}|V(x) - V_{\theta}(x)| dx  \le \frac{1}{\varepsilon} \int_{\mathbb{R}^N} |V(x) - V_{\theta}(x)| dx \leq \frac{C}{\varepsilon} < \infty.
		\end{eqnarray}
		The estimates above prove the desired result.
	\end{exm}
For the reader's convenience we consider also an example of potential $V$ which is asymptotically periodic where the set $V^{-1}(0)$ is not empty. In fact, we need to show the existence of an example for this kind of potential in such way that hypothesis $(V_3)$ is satisfied. For this purpose, we need to introduce the well-known Gagliardo–Nirenberg inequality for the fractional $p$-Laplacian operator. This inquality is motived by \cite[Theorem 1.4]{tue} which can be stated as follows: 
 \begin{proposition}\label{gagl} Let $s \in (0, 1)$ and $N > sp$. Then there exists a positive constant $C:= C(N, s, p)$ such that for any $u \in W^{s, p}(\mathbb{R}^N)$
 $$\|u\|^{p^*_s}_{p^*_s} \leq C \int_{\mathbb{R}^N} \int_{\mathbb{R}^N} \dfrac{|u(x) - u(y)|}{|x - y|^{N + sp}} dx dy, u \in W^{s,p}(\mathbb{R}^n).$$
 \end{proposition}
 Hence, we consider the following example:
\begin{exm}\label{120} Let $V: \mathbb{R}^N \mapsto \mathbb{R}$ be a potential given by $$V(x) = 1 - \dfrac{1}{\left(1 + |x|\right)^{\rho}}, x \in \mathbb{R}^N$$ where 
$\rho > N/r$, $r =p^*_s/(p^*_s - p)$. In particular, $V(x) \to 1$ as $|x| \to \infty$, that is, $V$ is asymptotically periodic. Moreover $V^{-1} (0) = \left\{x \in \mathbb{R}^N: V(x) = 0\right\} = \{0\}$. Now, we claim that $$\sigma = \sigma(V) = \inf_{u \in X} \left\{\|u\|^p: \|u\|_p^p = 1\right\} > 0.$$
The proof for this claim follows arguing by contradiction. Suppose that there exists a sequence $(u_k) \subset X$ such that $\|u_k\|_p^p = 1$ for all $k \in \mathbb{N}$ and $\|u_k\|^p \to 0$ as $k \to \infty$. Hence, by using \eqref{norma}, we infer that 
$$[u_k]^p \to 0,\ \int_{\mathbb{R}^N}V(x)|u_k|^p dx \to 0.$$
According to the Gagliardo-Nirenberg inequality, we obtain that $\|u_k\|_{p^*_s}^{p^*_s} \leq C [u_k]^{p^*_s}$. It follows that $\|u_k\|_{p^*_s} \to 0$ as $k \to \infty$. Therefore, $u_k \to 0$ a. e. in $\mathbb{R}^N$. Now, we see that
\begin{equation}\label{ck}
1 = \int_{\mathbb{R}^N}|u_k|^p dx = \int_{\mathbb{R}^N}\left[1 - \dfrac{1}{(1 + |x|)^{\rho}}\right]|u_k|^p dx + \int_{\mathbb{R}^N}\dfrac{|u_k|^p}{(1 + |x|)^{\rho}} dx = o_k(1) + \int_{\mathbb{R}^N}\dfrac{|u_k|^p}{(1 + |x|)^{\rho}} dx  
\end{equation}
Now, we shall estimate the integral in the right-hand side of the identity given just above. Notice that for each $R > 0$ fixed, we obtain that
\begin{eqnarray} \nonumber \displaystyle \int_{\mathbb{R}^N}\dfrac{|u_k|^p}{(1 + |x|)^{\rho}} dx = \int_{B_R(0)}\dfrac{|u_k|^p}{(1 + |x|)^{\rho}}dx + \int_{\mathbb{R}^N\setminus B_R(0)}\dfrac{|u_k|^p}{(1 + |x|)^{\rho}}dx 
\end{eqnarray}
Define $$\Omega = B_R(0), A_k = \int_{\Omega}\dfrac{|u_k|^p}{(1 + |x|)^{\rho}}dx\ \hbox{and}\ B_k = \int_{\mathbb{R}^N\setminus \Omega}\dfrac{|u_k|^p}{(1 + |x|)^{\rho}}dx.$$
Now, by using the Poincaré inequality \cite[Lemma 2. 4]{poin}, we can check that $\|u_k\|^p_{L^p(\Omega)} \leq C [u_k]^p_{\Omega}$ where
$$[u_k]^p_{\Omega} = \int_{\Omega} \int_{\Omega} \dfrac{|u_k(x) - u_k(y)|^p}{|x - y|^{N + sp}} dx dy.$$
Therefore, we see that
\begin{equation}\label{ak}A_k \leq \int_{\Omega} |u_k|^p dx \leq C[u_k]_{\Omega}^p \leq C[u_k]^p \to 0,\end{equation}
as $k \to \infty$. On the other hand, we mention that 
\begin{equation}\label{bkk} B_k = \int_{\mathbb{R}^N \setminus \Omega} \dfrac{|u_k|^p}{(1 + |x|)^{\rho}} dx = \int_{\left[\mathbb{R}^N \setminus \Omega\right] \cap [|u_k| \geq 1]} \dfrac{|u_k|^p}{(1 + |x|)^{\rho} } dx + \int_{\left[\mathbb{R}^N \setminus \Omega\right] \cap [|u_k| \leq 1]} \dfrac{|u_k|^p}{(1 + |x|)^{\rho} } dx    
\end{equation}
Since $1 < p < p^*_s$ it follows from the Gagliardo-Nirenberg inequality that
\begin{equation}\label{bk}
\int_{\left[\mathbb{R}^N \setminus \Omega\right] \cap [|u_k| \geq 1]} \dfrac{|u_k|^p}{(1 + |x|)^{\rho} } dx \leq \int_{\mathbb{R}^N} |u_k|^{p^*_s} dx \leq C [u_k]^{p^*_s}\end{equation}
Furthermore, taking into account the Hölder's inequality with exponents $r = p^*_s/(p^*_s - p)$ and $p^*_s/p$, we deduce that
\begin{equation}\label{bk2} \int_{\left[\mathbb{R}^N \setminus \Omega\right] \cap [|u_k| \leq 1]} \dfrac{|u_k|^p}{(1 + |x|)^{\rho} } dx \leq \left(\int_{\left[\mathbb{R}^N \setminus \Omega\right] \cap [|u_k| \leq 1]} \dfrac{1}{(1 + |x|)^{r\rho} }dx\right)^{\dfrac{1}{r}}\left( \int_{\left[\mathbb{R}^N \setminus \Omega\right] \cap [|u_k| \leq 1]} |u_k|^{p^*_s}dx\right)^{\frac{p}{p^*_s}}.
\end{equation}
Now, by using the Gagliardo-Nirenberg inequality and the Co-area Theorem, we obtain that
\begin{equation}\nonumber \displaystyle \int_{\left[\mathbb{R}^N \setminus \Omega\right] \cap [|u_k| \leq 1]} \dfrac{|u_k|^p}{(1 + |x|)^{\rho} } dx \leq \tilde{C} [u_k]^p  \int_{R}^{\infty} \ell^{N - 1 - \rho r} d\ell = \tilde{C}[u_k]^p  \dfrac{1}{R^{\rho r - N} (\rho r - N)} = C_R [u_k]^p. 
\end{equation}
Here we observe that was used the inequality $\rho > N/r$ in order to ensure that $C_R$ is finite. Now, taking into account that $[u_k] \to 0$, the last estimate together with \eqref{ak}, \eqref{bkk} and \eqref{ck}, imply that
$1 = \|u_k\|_p^p \to 0$ as $k \to \infty$. This is a contradiction which shows that $\sigma > 0$.
\end{exm}

Now, we consider an asymptotically periodic potential that vanishes on the set $W_a = \{ x \in \mathbb{R}^N: |x| = a\}$ which corresponds to the boundary of the set $B_a(0) = \{ x \in \mathbb{R}^N: |x| < a\}$. The last set is the open ball centered at zero with radius $a$. 
\begin{exm} Let us fix $a > 0$ and $\rho > p^*_s/(p^*_s -p)$. Consider $V: \mathbb{R}^N \mapsto \mathbb{R}$ defined as follows:
$$V(x) = \left\{\begin{array}{lll}
    1 - \left(\dfrac{a}{|x|}\right)^{\rho},\ |x| > a,\\
    0,\ |x| = a,\\
    \dfrac{-1}{2a} |x| + \dfrac{1}{2}, 0 \leq |x| < a.
\end{array}\right.$$
 Hence, $V \in C^0(\mathbb{R}^N, \mathbb{R})$, $V(x) \geq 0, x \in \mathbb{R}^N$. Furthermore, $V(x) = 0$ for all $x \in W_a = \{x \in \mathbb{R}^N: |x| = a\}$. Now, we claim that $(V_3)$ is satisfied. The proof follows arguing by contradiction. In fact, using the same ideas employed in the previous example, we assume that there exists a sequence $(u_k) \subset X$ such that $\|u_k\|_p^p = 1$ and $\|u_k\| \to 0$ as $k \to \infty$. Therefore, we obtain that
$$1 = \|u_k\|^p_p = \int_{\mathbb{R}^N}V(x) |u_k|^p dx + \int_{\mathbb{R}^N}[1 - V(x)]|u_k|^p dx = o_k(1) + \int_{\mathbb{R}^N}[1 - V(x)]|u_k|^p dx.$$
Our goal is to ensure that the last integral in right-hand side for the expression just above converges to zero as $k \to \infty$. In order to do that we write
$$\int_{\mathbb{R}^N} [1 - V(x)]|u_k|^p dx = \int_{B_a(0)} \left[\dfrac{-|x|}{2a} + \dfrac{1}{2}\right]|u_k|^p dx + \int_{\mathbb{R}^N \setminus B_a(0)}\dfrac{a^{\rho}|u_k|^p}{|x|^{\rho}}dx.$$
Notice that for the set $W_a$ the Poincaré inequality implies that
$$\int_{B_a(0)}\left[\dfrac{-|x|}{2a} + \dfrac{1}{2}\right]|u_k|^p dx = o_k(1).$$
Furthermore, by using Hölder's inequality together with the Co-area Theorem as was done in Example \ref{120}, we deduce that $$\int_{\mathbb{R}^N \setminus B_a(0)}\dfrac{a^{\rho}|u_k|^p}{|x|^{\rho}}dx \leq C[u_k]^p = o_k(1).$$
Under these conditions, we obtain that that $1 = \|u_k\|_p^p \to 0$. Hence, we obtain a contradiction proving that $\sigma = \sigma(V) > 0$.
\end{exm}

 In our next example, we shall exhibit an asymptotically periodic potential $V: \mathbb{R}^N \mapsto \mathbb{R}$ that vanishes on a set of positive Lebesgue measure. Furthermore, the potential $V$ given just below satisfies hypothesis $(V_3)$.
 \begin{exm} Consider $V:\mathbb{R}^N \mapsto \mathbb{R}$ given by $V(x) = 1 - g(x), x \in \mathbb{R}^N$ where
 $$g(x) = \left\{\begin{array}{lll} 1,\ |x| \leq 1,\\
 \dfrac{6^{\rho}}{(1 + |x|)^{\rho}(1 + 2|x|)^{\rho}},\ |x| > 1.
 \end{array}\right.$$
Notice that $V^{-1}(0) = \overline{B_1(0)}$. Furthermore, we see that $V \in C^0(\mathbb{R}^N, \mathbb{R})$. Once again we claim that $(V_3)$ is verified. The proof follows by contradiction. Let us assume that there exists a sequence $(u_k) \subset X$ such that $\|u_k\|_p^p = 1$ and $\|u_k\|^p \to 0$ as $k \to \infty$. Moreover, we mention that
$$\displaystyle \int_{\mathbb{R}^N}|u_k|^p dx = \int_{\mathbb{R}^N}[1 - g(x)]|u_k|^p + \int_{\mathbb{R}^N}g(x)|u_k|^p dx = o_k(1) + \int_{\mathbb{R}^N}g(x)|u_k|^p dx. $$
Under these conditions, we shall analyze the last integral on the right-hand side for the expression given just above. Recall that for the set $x \in \overline{B_1(0)}$ 
we can apply the Poincaré inequality. Hence, we obtain
$$\int_{B_1(0)}g(x)|u_k|^p dx = \int_{B_1(0)}|u_k|^p dx \leq C [u_k]^p_{\Omega} \leq C[u_k]^p.$$
Once again the last intergral goes to zero as $k \to \infty$. It remains to consider the integral outside the set $\overline{B_1(0)}$. Let us define $$C_k = \int_{[\mathbb{R}^N\setminus B_1(0)] \cap |u_k| \geq 1}g(x)|u_k|^p dx\ \hbox{and}\  D_k = \int_{[\mathbb{R}^N\setminus B_1(0)] \cap |u_k| \leq 1}g(x)|u_k|^p dx$$
Since the function $g$ is bounded from above and taking into acount that $p < p^*_s$, we infer that
$$C_k \leq \int_{[\mathbb{R}^N\setminus B_1(0)] \cap |u_k| \geq 1}|u_k|^p dx \leq C\int_{\mathbb{R}^N}|u_k|^{p^*_s} dx \leq \tilde{C} [u_k]^{p^*_s}.$$
Hence, using once more that $[u_k] \to 0$, we obtain that $C_k \to 0$ as $k \to \infty$. On the other hand, by using Hölder's inequality with $r=p^*_s/(p^*_s - p)$ e $p^*_s/p$, we check that
$$D_k\leq \int_{[\mathbb{R}^N\setminus B_1(0)] \cap |u_k| \leq 1} \dfrac{6^{\rho}|u_k|^p}{(1 + |x|)^{\rho}} dx \leq C \left(\int_{\left[\mathbb{R}^N \setminus \Omega\right] \cap [|u_k| \leq 1]} \dfrac{1}{(1 + |x|)^{r\rho} }dx\right)^{\dfrac{1}{r}}\left( \int_{\left[\mathbb{R}^N \setminus \Omega\right] \cap [|u_k| \leq 1]} |u_k|^{p^*_s}dx\right)^{\frac{p}{p^*_s}}.$$
Thus, by using the Co-area Theorem and the fact that $[u_k] \to 0$ as $k \to \infty$, we can argue in the same way as was done in Example \ref{120} proving that $D_k \to 0$ as $k \to \infty$. Hence, we obtain that
$1 = \|u_k\|_p^p \to 0$ as $k \to \infty$. The last statement does not make sense proving that $\sigma > 0$.
\end{exm}

   Now, we consider the sphere of radius $m$ in space $L^p(\mathbb{R}^N)$ as follows:
	\begin{equation}\label{S} S_m = \left\{u \in X: \|u\|_p^p = m^p \right\}.\end{equation}
	Our main objective is to find the existence of minimizers for the functional $J: X \rightarrow \mathbb{R}$ restricted to the set $S_m$, that is, we ensure the existence of $u \in S_m$ that satisfies
	\begin{equation}\label{minimization}
		\gamma_m = \inf\{J(w): w \in S_m\} = J(u),
	\end{equation}
	where the energy functional $J: X \to \mathbb{R}$ is given by
	$$J(u) = \dfrac{1}{p}[u]^p + \dfrac{1}{p}\int_{\mathbb{R}^N}V(x)|u|^p dx - \dfrac{\beta}{q}\int_{\mathbb{R}^N}|u|^q dx,\ \  \forall u\in X.$$
	It is worthwhile to mention that $J$ is in $C^1$ class restrict to $S_m$. Furthermore, given a minimizer $u \in S_m$ for the minimization Problem \eqref{minimization}, we obtain that $u$ is a weak solution to the Problem \eqref{P}. The main idea in the last assertion is to apply the Lagrange Multiplier Theorem. 
	
	Now, we shall state our main results. Firstly, for the $L^p$-subcritical case with periodic potential, we consider the following result:
	
	\begin{theorem}[$L^p$-subcritical case, periodic potential]\label{TB1} Suppose that $q \in (p, p + sp^2/N)$, $(V_1)$ and $\beta > 0$.  Then, for every $m > 0$, there exists $\delta=\delta(m) > 0$ such that if $\|V\|_{\infty} < \delta$, we obtain that the Problem \eqref{P} has at least one positive weak solution $u \in S_m$ satisfying $J(u) = \gamma_{m} < 0$. 
 Furthermore, we also obtain that $u \in L^{\infty}(\mathbb{R}^N) \cap C_{loc}^{0,\alpha}(\mathbb{R}^N)$ holds true for some $\alpha \in (0,1)$.
	\end{theorem}
	
	Now, by using hypothesis $(V_2)$, we can consider the existence of local minimizers $u \in S_m$ for
	asymptotically periodic potentials. More precisely, we are able to consider the following result;
	\begin{theorem}[$L^p$-subcritical case, asymptotically periodic potential]\label{TB2} Suppose $q \in (p, p + sp^2 / N)$, $\beta > 0$, $(V_2)$ and $(V_3)$. Then, for every $m > 0$ there exists $\delta=\delta(m) > 0$ such that if $\|V_{\theta}\|_{\infty} < \delta $, the Problem \eqref{P} has at least one positive weak solution $u \in S_m$ such that $J(u) = \gamma_{m} < 0$. Furthermore, we also obtain that $u \in L^{\infty}(\mathbb{R}^N) \cap C_{loc}^{0,\alpha}(\mathbb{R}^N)$ holds true for some $\alpha \in (0,1)$.
	\end{theorem}
	
	Now, motivated in part by \cite{Rab}, we consider a crucial hypothesis in order to prove our third main result. Namely, we consider the following assumption:
	\begin{itemize}
		\item[($V_4$)] $V \in C(\mathbb{R}^N) \cap L^{\infty}(\mathbb{R}^N)$ and $$V_{\infty} = \liminf_{|x|\to \infty} V (x) > V_0 = \inf_{x \in \mathbb{R}^N} V(x) > 0.$$
	\end{itemize}
	Furthermore, we define the functional $J_\varepsilon : X \longrightarrow \mathbb{R}$, given by
	\begin{equation}\label{jeps}
		J_{\varepsilon}(u) = \frac{1}{p}[u]^p + \frac{1}{p}\int_{\mathbb{R}^N} V(\varepsilon x)|u|^p dx - \frac{\beta}{q}\int_{\mathbb{R}^N} |u|^q dx.\end{equation}
	As a consequence, we consider the following minimization problem:
	\begin{equation}\label{gammameps} \gamma_{m, \varepsilon} = \inf\left\{J_{\varepsilon}(u): u \in S_m\right\}.
 \end{equation}
At this stage, we prove that there exists a solution for the Problem \eqref{gammameps}. The main idea is to employ the Lagrange Multiplier Theorem \cite{Dra}. In other words, we obtain a weak solution to the following nonlocal elliptic problem:
	\begin{equation}\label{Peps}\tag{$P_{m,\varepsilon}$}
		\left\{\begin{array}{cc}
			\displaystyle (-\Delta)^s_p u\ +\ V (\varepsilon x) |u|^{p-2}u\  = \lambda |u|^{p - 2}u + \beta\left|u\right|^{q-2}u\ \hbox{in}\ \mathbb{R}^N, \\
			\displaystyle \|u\|_p^p = m^p,\ u \in W^{s, p}(\mathbb{R}^N),\end{array}\right.
	\end{equation}
where $\varepsilon > 0$ is a small enough fixed parameter, $\beta > 0$ and  $\lambda > 0$.
	
	\begin{theorem}[$L^p$-subcritical case]\label{TB3} Assume that $q \in (p, p + sp^2/N)$, $\beta > 0, (V_3)$ and  $(V_4)$. Then, for every $m > 0$ there exists $\delta=\delta(m) > 0$ and $\varepsilon_0 > 0$  such that if $\|V\|_{\infty} < \delta$,  $\gamma_{m , \varepsilon} < 0$. Furthermore, there exists $u \in S_m,$ $u > 0$ such that $u$ is a weak solution for the Problem \eqref{Peps}, for all $\varepsilon \in (0, \varepsilon_0)$. Furthermore, we also obtain that $u \in L^{\infty}(\mathbb{R}^N) \cap C_{loc}^{0,\alpha}(\mathbb{R}^N)$ holds true for some $\alpha \in (0,1)$.
	\end{theorem}
	In our last result we consider the case where the problem has a $L^p-$critical growth, that is, we assume that $q = p + sp^2/N$. Under this condition, by using a hypothesis introduced in \cite{bart}, we consider the following assumption:
	\begin{itemize}
		\item[($V_5$)] There holds $\mu\left(\{x \in \R^N : V(x) \leq M \}\right) < \infty$ for each $M > 0$.
	\end{itemize}
	\begin{theorem}[$L^p$-critical case]\label{TB4} 
		Suppose $q = p + sp^2/N$. Also assume that $(V_3)$ and $(V_5)$ and $V \geq 0$ are satisfied. Then, there exists $\beta_0>0 $ such that  the Problem \eqref{P} has at least one positive weak solution $u \in S_m$, for every $\beta \in (0, \beta_0)$ where $\beta_0 = \beta_0(N, s,p,m) > 0$.	Furthermore, we also obtain that $u \in L^{\infty}(\mathbb{R}^N) \cap C_{loc}^{0,\alpha}(\mathbb{R}^N)$ holds true for some $\alpha \in (0,1)$.
	\end{theorem}

	\subsection{Notation} In the present work we shall use the following notations:
	\begin{itemize}
		\item $K, K_1, \cdots,$ denotes the positive constants.
		\item  The norm in $L^t(\mathbb{R}^N)$ and $L^{\infty}(\mathbb{R}^N)$ will be denote by $\|\cdot\|_t, t \in [1, \infty]$.
		\item $\|\cdot\|$ denotes the norm of working space $X$.
		\item $[\ \cdot\ ]$ denotes the Gagliardo seminorm.
		\item The open ball in $\mathbb{R}^N$ centered at $x \in \mathbb{R}^N$ with radius $r > 0$ is denoted by $B_r(x)$.
		\item Given $A \subset \mathbb{R}^N$ we define $A^c = \{x \in \mathbb{R}^N: x \notin A\}$.
		\item Given any measurable set $A \subset \mathbb{R}^N$ the Lebesgue measure of $A$ is denoted by $\mu(A)$.
	\end{itemize}
	\subsection{Outline} 
	The remainder of this paper is organized as follows: In the forthcoming section we consider some preliminaries results involving the fractional $p$-Laplacian operator. In Section 3 we consider some results related to the functional $J$ proving the Theorem \ref{TB1} where $V$ is $1$-periodic. In Section 4 we
	consider some strong results proving Theorem \ref{TB2} where the potential $V$ is asymptotically periodic. In Section 5 we employ Theorems \ref{TB1} and \ref{TB2} in order to prove Theorem \ref{TB3}. The last section is devoted to the case $q = p + sp^2/N$ giving sufficient conditions for the existence of at least one nontrivial solution for the Problem \eqref{P}.

	\section{Preliminary Results}\label{SB1}
	Now, we shall consider the existence of critical points with prescribed norm for the functional $J: X \longrightarrow \mathbb{R}$ given by
	\begin{equation}\label{JB}J(u) = \frac{1}{p}[u]^p + \frac{1}{p}\int_{\mathbb{R}^N} V(x)|u|^p dx - \frac{\beta}{q}\int_{\mathbb{R}^N} |u|^q dx,\ \ \ u\in X.\end{equation}
	Here we emphasize once again that
	$$X = \left\{u \in W^{s,p}(\mathbb{R}^N): \int_{\mathbb{R}^N} V(x)|u|^p dx < \infty\right\}.$$  
Recall that the usual norm for the space $W^{s, p}(\mathbb{R}^N)$ is given by $$\|u\|_{*} = ([u]^p + \|u\|^p_p)^{\frac{1}{p}}, u \in  W^{s, p}(\mathbb{R}^N).$$
  For the next result we shall prove that the norm $\|\cdot \|$ and the usual norm in $X$ given just above are equivalents. Namely, we shall prove the following result:
	
\begin{proposition}\label{ob11} 
Suppose that $V \in L^{\infty}(\mathbb{R}^N), V \geq 0, V \not\equiv 0$. Assume also that $(V_3)$ is satisfied. Then the norm $\|\cdot\|$ is equivalent to the usual norm of the space $W^{s, p}(\mathbb{R}^N)$.
\end{proposition}

 \begin{proof}
 Firstly, by using the fact that $V \in L^{\infty}(\mathbb{R}^N)$, we obtain that 
 $$\|u\|^p \leq [u]^p + \|V\|_{\infty}\|u\|_p^p \leq c_1 \|u\|^p_{*},  u \in  W^{s, p}(\mathbb{R}^N)$$
 where $c_1 = \max\{1, \|V\|_{\infty}\}$. In virtue of hypothesis $(V_3)$ we infer also that $$\|u\|_p^p \leq \frac{1}{\sigma} \|u\|^p, u \in  W^{s, p}(\mathbb{R}^N).$$ As a consequence, we obtain that $$\|u\|^p_{*} = [u]^p + \|u\|_p^p \leq [u]^p + \frac{1}{\sigma} \|u\|^p \leq \|u\|^p + \frac{1}{\sigma} \|u\|^p = c_2 \|u\|^p,  u \in  W^{s, p}(\mathbb{R}^N)$$
 where $c_2 = 1 + 1/\sigma$. This completes the proof.
\end{proof}
\begin{proposition}\label{reflexive} Assume $(V_3)$ and $V \in L^{\infty}(\mathbb{R}^N), V \geq 0, V \not\equiv 0$. Then $X$ endowend with the norm $\|\cdot\|$ is reflexive.
\end{proposition}
\begin{proof}
Fisrstly, we consider the Banach space $E = L^p(\mathbb{R}^N) \times L^p(\mathbb{R}^N \times \mathbb{R}^N)$. Define the map $T: W^{s, p}(\mathbb{R}^N) \mapsto E$ in the folowing way that $T(u) = (u, v)$ where $$v = v(x, y) = \dfrac{u(x) - u(y)}{|x - y|^{\frac{N}{p} + s}}, x , y \in \mathbb{R}^N$$ for each function $u \in W^{s, p}(\mathbb{R}^N)$. 

It is easy to verify that $T$ is a well-defined linear map. Furthermore, we see that $T(W^{s, p}(\mathbb{R}^N))$ is closed using the norm $\|\cdot\|$. Indeed, consider a sequence $(u_k)$ in $W^{s, p}(\mathbb{R}^N)$ such that $\lim_{k \to \infty} T(u_k) = \lim_{k \to \infty}(u_k, v_k) = (f, g) \in E$.  Under these conditions, by using the fact that $u_k \to f$ in $L^p(\mathbb{R}^N)$, we obtain that $u_k \to f$ a. e. in $\mathbb{R}^N$. Similarly, $v_k \to g$ a. e. in $\mathbb{R}^N$. Therefore, we deduce that
$$g(x, y) = \lim_{k \to \infty} \dfrac{u_k(x) - u_k(y)}{|x - y|^{\frac{N}{p} + s}} = \dfrac{f(x) - f(y)}{|x - y|^{\frac{N}{p} + s}},\ \hbox{a. e. in}\ \mathbb{R}^N.$$
It follows that $(f, g) \in T(W^{s, p}(\mathbb{R}^N))$. This assertion proves that $T(W^{s,p}(\mathbb{R}^N))$ is closed with the usual norm. Moreover, the usual norm is equivalent to the norm $\|\cdot\|$, see Proposition \ref{ob11}. Hence, $T(W^{s,p}(\mathbb{R}^N))$ is a closed set using the norm $\|\cdot\|$.

Now, we define $\Phi = i \circ T: W^{s, p}(\mathbb{R}^N) \mapsto E$ where $i(T(u)) = i(u,v) = (u,v)$ for each $u \in W^{s, p}(\mathbb{R}^N)$. Here we look for $E$ using the usual norm. Therefore, by using the fact that $T(W^{s, p}(\mathbb{R}^N))$ is closed under the norm $\|\cdot\|$, it follows that $i\left(T(W^{s, p}(\mathbb{R}^N))\right) \subset E$ is closed. Here was used that $T$ is continuous. Recall also that $E$ is a reflexive Banach space. Hence, $i\left(T(W^{s, p}(\mathbb{R}^N))\right) = T(W^{s, p}(\mathbb{R}^N))$ is reflexive. Furthermore, $i^{-1}: i\left(T(W^{s, p}(\mathbb{R}^N))\right) \mapsto T(W^{s, p}(\mathbb{R}^N))$ is a linear and continuous map. Now, we are able to use  that $i\left(T(W^{s, p}(\mathbb{R}^N))\right)$ is reflexive. In fact, using the last assertion, we obtain that $i^{-1}\left(i\left(T(W^{s, p}(\mathbb{R}^N))\right)\right) = T(W^{s, p}(\mathbb{R}^N))$ is reflexive using the norm $\|\cdot\|$. This ends the proof.
\end{proof}
	Now, by using the interpolation law for $L^t(\mathbb{R}^N)$ we obtain that 
	\begin{equation}\label{pp}
		\|u\|^q_q \le K_1\|u\|_p^{(1-t)q} \|u\|^{tq}_{p^*_s},
	\end{equation} 
	for some $t \in (0, 1)$.
 
Hence, by using Proposition \ref{gagl} and \eqref{pp} we obtain that 
\begin{equation}\label{B1}
 \|u\|^q_q \leq C\|u\|_p^{(1-t)q}[u]^{tq},\ t = N\left(\frac{1}{sp} - \frac{1}{sq}\right).
 \end{equation}
The previous estimate will be used in order to control the behavior of the functional $J$ restricted to the sphere $S_m$.

	\section{$L^p -$ Subcritical Growth: The periodic case}\label{SecB1}
	
	In this section, as a first step in order to establish the existence of a critical point for the functional $J$ we shall prove that $J$ is bounded from below in $S_m$.

	\begin{proposition}\label{lim} Suppose $V \in L^{\infty}(\mathbb{R}^N)$, $V \geq 0$, $\beta > 0$. Assume also that $q \in (p, p + sp^2/N)$. Let $(u_k) \subset S_m$ be a minimizing sequence to the minimization problem given in \eqref{minimization}. Then, $(u_k)$ is a bounded sequence in $X$.    
	\end{proposition}
	
	\begin{proof} Initially, we shall prove that $J$ is bounded from below in $S_m$. In fact, for each $u \in S_m$ and \eqref{JB}, we observe that
		$$J(u) = \frac{1}{p}[u]^p + \frac{1}{p}\int_{\mathbb{R}^N} V(x)|u|^p dx - \frac{\beta}{q}\int_{\mathbb{R}^N} |u|^q dx.$$
		Now, by using the \eqref{B1} and the fact that $V$ is nonnegative potential, we obtain
		\begin{equation*} J(u) \geq \frac{1}{p}[u]^p - \frac{\beta C}{q} \|u\|_p^{(1 - t )q}[u]^{tq}.
		\end{equation*}
		Therefore, we obtain that
		\begin{equation}\label{2}
			J(u) \geq \frac{1}{p} [u]^p - \frac{\beta C}{q}m^{(1 - t)q} [u]^{tq}.
		\end{equation}
		Since $t = N\left(\frac{1}{sp} - \frac{1}{sq}\right)$ and $q \in (p, p + sp^2/N )$, we conclude that $tq < p$. Hence, by using \eqref{2}, we obtain that functional $J$ is bounded from below. 
		
		From now on, consider a sequence $(u_k) \subset S_m$ be a minimizing sequence for the minimization Problem \eqref{minimization}, that is, $\gamma_{m} + o_k(1)= J(u_k)$. It is not hard to see that 
		$$\displaystyle J(u_k) \geq \frac{1}{p}\|u_k\|^p - \frac{\beta C m^{(1 - t)q}}{q} [u_k]^{tq},$$ where $t = N \left(\frac{1}{sp} - \frac{1}{sq}\right)$. Now, we argue by contradiction  assuming that $(u_k)$ is unbounded in $X$.  Hence, by using the fact that $q \in (p, p + sp^2/N)$, we deduce that $\gamma_{m} =\displaystyle \lim_{k \to \infty} J(u_k) = \infty$. This is a contradiction proving that $(u_k)$ is bounded in $X$. This ends the proof.
	\end{proof}
	It is important to stress that $- \infty < \gamma_m =\displaystyle \inf\{J(u): u \in S_m\} < \infty$. As a consequence, the minimization problem given in \eqref{minimization} is well defined. Furthermore, we can consider the following result:

	\begin{lemma}\label{neg}
		Assume that $V \in L^{\infty}(\mathbb{R}^N), V \geq 0, \beta > 0$ and $q \in (p, p + sp^2/ N)$. Then, for every $m > 0$ there exists $\delta = \delta(m) > 0$ such that $\gamma_m < 0$ for each $\|V\|_{\infty} < \delta$.
	\end{lemma}
	\begin{proof} Let us consider the function $t \mapsto u_t$ where $u_t(x) = t^{N/p}u(tx), u \in S_m, t > 0$. In this case, using the change of variables $y = tx$ we observe that
		$$\displaystyle \int_{\mathbb{R}^N} |u_t(x)|^p dx = \int_{\mathbb{R}^N} |u|^p dy = m^p.$$ The last assertion ensures that $u_t \in S_m$ for all $t > 0$. Similarly, we mention that
		$$\displaystyle \int_{\mathbb{R}^N} |u_t|^q dx = t^{\frac{N}{p}(q - p)}\int_{\mathbb{R}^N} |u| ^q dy.$$ On the other hand, we mention that
		$$[u_t]^p = \int_{\mathbb{R}^{N}}\int_{\mathbb{R}^{N}} \frac{|u_t(x) - u_t(y)|^p }{|x - y|^{N + sp}}dx dy.$$
		
		Now, for $v = tx$ and $z = ty$, we check also that
		$[u_t]^p = t^{sp}[u]^p$.
		Hence, we obtain that
		\begin{eqnarray}\label{3}\nonumber \displaystyle J (u_t) &=& \frac{1}{p} [u_t]^p + \frac{1}{p}\int_{\mathbb{R }^N} V(x)|u_t|^p dx - \frac{\beta}{q}\int_{\mathbb{R}^N} |u_t|^qdx\\
			\nonumber \displaystyle &=& \frac{t^{sp}}{p}[u]^p + \frac{1}{p}\int_{\mathbb{R}^N}V\left(\frac{x}{t}\right)|u|^p dx - \frac{\beta t^{ \frac{N}{p}(q - p)}}{q} \int_{\mathbb{R}^N} |u|^q dx\\
			\displaystyle &\le& \frac{t^{sp}}{p}[u]^p + \frac{\|V\|_{\infty}}{p}m^p - \frac{\beta t^{\frac{ N}{p}(q - p)}}{q} \int_{\mathbb{R}^N} |u|^q dx.\end{eqnarray}
		Now, we define $$\displaystyle R_t = \frac{t^{sp}}{p}[u]^p - \frac{\beta t^{\frac{N}{p}(q - p)}}{q} \int_{\mathbb{R}^N}|u|^q dx.$$
		We claim that there exists  $t > 0$ small enough such that $R_t < 0$. Indeed, by using the fact that $q \in (p, p + sp^2/N)$, we obtain $\displaystyle sp >  (q - p)N/p$. Hence, by using that $t \mapsto R_t$ is continuous, we observe that for $t > 0$ sufficiently small $R_t < 0$.  Furthermore, we choose $\delta = -R_t/m^p$. Thus, assuming that $\|V\|_{\infty} < \delta$ and taking into account \eqref{3}, we deduce that
		\begin{equation}
			J(u_t) \le R_t + \frac{\|V\|_{\infty}}{p} m^p < R_t - \frac{R_t}{p} = \left(\frac{p - 1} {p}\right)R_t < 0.\end{equation}
		As a consequence, $\gamma_m < 0$ is verified. This ends the proof. 
	\end{proof}
	
	Now, we shall prove a technical result that allows us to show the strong convergence for a specific minimizing sequence. Here we mention that the compact embedding from the Sobolev spaces into the Lebesgue space is not satisfied in our framework. In order to overcome this difficulty we consider the following minimization problems:
	\begin{equation}\label{Pmi} 
		\gamma_{m_i} = \inf\{J(u) : u \in S_{m_i}\},\ \ \ S_{m_i} = \left\{u \in X: \|u\|_p^p  = m_i^p \right\},\ \ 
		i=1,2.
	\end{equation}	
Hence, we can formulate the following result:

	\begin{lemma}\label{important} Suppose that $V$ satisfies $(V_1), \beta > 0$ and $q \in (p, p + sp^2/N).$ Assume also that $\|V\|_{\infty} < \delta$ is satisfied. Then, for each $0 < m_1 < m_2$ we obtain that $$\frac{m_1^p}{m_2^p} \gamma_{m_2} < \gamma_{m_1}.$$
	\end{lemma}
	\begin{proof} Let $\xi > 1$ be in such way that  $m_2 = \xi m_1$. Clearly, there exists $\xi$ due to the fact that that $m_1 < m_2$.  Let us consider also a sequence $(u_k) \subset S_{m_1}$ a minimizer sequence for the minimization Problem \eqref{Pmi} with $i =1$, that is, $\displaystyle \lim_{k \to \infty} J(u_k) = \gamma_{m_1}$. Define also the sequence $v_k = \xi u_k$. Now, we claim that $v_k \in S_{m_2}$. In fact, we observe that
		$$\int_{\mathbb{R}^N} |v_k|^p dx = \xi^p \int_{\mathbb{R}^N} |u_k|^p dx = \xi^p m_1^p = m^p_2.$$
		Since $\displaystyle \gamma_{m_2} = \inf \{J(u): u \in S_{m_2}\}$, we infer that
		$$\gamma_{m_2} \le J(v_k) = \frac{\xi^p}{p}[u_k]^p + \frac{\xi^p}{p}\int_{\mathbb{R}^N} V(x)|u_k|^p dx - \frac{\xi^q \beta}{q}\int_{\mathbb{R}^N} |u_k|^q dx. $$
		Now, by using the term $\displaystyle \frac{\xi^p \beta}{q} \int_{\mathbb{R}^N} |u|^q dx$, we deduce that 
		\begin{equation}\label{4}
			\gamma_{m_2} \le \xi^p J(u_k) + \frac{\beta(\xi^p - \xi^q)}{q} \int_{\mathbb{R}^N}|u_k|^q dx. 
		\end{equation}
		At this stage, we claim that there exists a constant $K > 0$ and $k_0 \in \mathbb{N}$ such that
		\begin{equation}\label{5}
			\int_{\mathbb{R}^N} |u_k|^q dx \geq K,\ k \geq k_0.
		\end{equation}
		Indeed, arguing by contradiction we assume that $$\int_{\mathbb{R}^N} |u_k|^q dx \to 0,\ k \to \infty.$$
		Hence, by using the fact that $(u_k)$ is a minimizing sequence using the parameter $\gamma_{m_1}$, we see that
		$$0 > \gamma_{m_1} + o_k(1) = J(u_k) \geq - \frac{\beta}{q} \int_{\mathbb{R}^N} |u_k|^q dx.$$
		Hence, doing $k \to \infty$, we arrive at a contradiction.
		Therefore, the proof of the desired claim follows. 
		Under these conditions, by using the estimate just above together with the fact that $q > p$, $\xi > 1$ and \eqref{4} we obtain $$\gamma_{m_2} \leq \xi^p J(u_k) + \frac{\beta (\xi^p - \xi^q)}{q}K.$$ Furthermore, doing $k \to \infty$, we also obtain that
		$$\gamma_{m_2} \le \xi^p \gamma_{m_1} + \frac{\beta(\xi^p - \xi^q)}{q}K < \xi^p \gamma_{m_1}.$$ As a consequence, we deduce that $\gamma_{m_2}/\xi^p < \gamma_{m_1}$. Recall also that $m_2 =\xi m_1$ which implies that $m_1^p\gamma_{m_2}/m_2^p < \gamma_{m_1}$. This completes the proof.
	\end{proof}

Now, we shall consider a Lions type result. Namely, we consider the folowing result:
\begin{lemma}\label{25} Assume that $N > sp$ and $r \in [p, p^*_s)$. Suppose also that $(u_k)$ is a bounded sequence in $W^{s, p}(\mathbb{R}^N)$ and
$$\lim_{k \to \infty} \left(\sup_{y \in \mathbb{R}^N} \int_{B_R(y)} |u_k|^r dx\right) = 0,$$ 
where $R > 0$. Then $u_k \to 0$ in $L^t(\mathbb{R}^N)$ for all $t \in (p, p^*_s)$.
\end{lemma}	
	\begin{proof}
	    These results can be found in \cite[Lemma 2.1]{ambi}.  We omit the details. 
	\end{proof}

	Now, we shall prove that we can choose, among all minimizing sequences, one minimizer sequence which weakly converges to a nonzero function. Namely, we consider the following result:
	
	\begin{lemma}\label{nontrivial} Suppose that $V$ satisfies $(V_1), \beta > 0$ and  $q \in (p, p + sp^2/N)$. Then there exists a minimizing sequence $(\tilde{u}_k) \subset S_m$ such that $\tilde{u}_k \rightharpoonup \tilde{u}$ where $\tilde{u} \not\equiv 0$.    
	\end{lemma}
	\begin{proof} Let $(u_k) \subset S_m$ be a minimizing sequence for the Problem \eqref{minimization}. We claim that there exists $\alpha, \eta > 0$ and a sequence $(y_k) \subset \mathbb{R}^N$ such that
		$$\int_{B_{\alpha}(y_k)}|u_k|^p dx \geq \eta, \, k \in \mathbb{N}.$$
		Once again the proof for this claim follows arguing by contradiction. Let us assume that 
		$$\displaystyle \int_{B_{\alpha}(y_k)}|u_k|^p dx \to 0.$$ 		
		Hence, by using Lemma \ref{25} and $q > p$, we mention that $$\displaystyle \int_{\mathbb{R}^N} |u_k|^q dx \to 0.$$ Furthermore, by using that $(u_k)$ is a minimizing sequence, we obtain a contradiction by using \eqref{5}.
		
		Now, we shall consider $(y_k) \subset \mathbb{Z}^N$. Define the auxiliary sequence $\tilde{u}_k = u(x + y_k)$. Therefore, by using the hypothesis $(V_1)$ together with the fact that $(y_k) \subset \mathbb{Z}^N$, we obtain that $\|\tilde{u}_k\| = \|u_k\|$. Hence, $(\tilde{u}_k)$ is now bounded in $X$. Up to a subsequence there exists a function $\tilde{u} \in X$ such that $\tilde{u}_k \rightharpoonup \tilde{u}$ in $X$. Hence, $\tilde{u}_k \to \tilde{u}$ in $L^{p}_{loc}(\mathbb{R}^N)$ and $\tilde{u}_k \to \tilde{u}$ a. e. in $\mathbb{R}^N$. As a consequence, we infer that
		$$\displaystyle \limsup_{k \to \infty} \int_{B_{\alpha}(0)} |\tilde{u}_k|^p dx \geq \eta > 0.$$ 
		Therefore, by using the Dominated Convergence Theorem, we obtain that 
		$$\int_{B_{\alpha}(0)} |\tilde{u}|^p dx \geq \eta > 0.$$
		Hence, $\tilde{u} \not\equiv 0$ is now verified. This ends the proof.
	\end{proof}
	
	\begin{lemma}\label{6} Assume $V \in L^{\infty}(\mathbb{R}^N), V \geq 0, \not\equiv 0$, $\|V\|_{\infty} < \delta, \beta > 0$ and $q \in (p, p + sp^2/N)$. Let $(u_k) \subset S_m$ be a minimizing sequence for the Problem \eqref{minimization}. Suppose that $u_k \rightharpoonup u$ in $X$ with $u \not\equiv 0$. Then $u \in S_m$.
		
	\end{lemma}
	
	\begin{proof} 
		The proof follows arguing by contradiction. Let us assume that $u \not\in S_m$. Hence, there exists $0<l \in \mathbb{R}\setminus\{m\}$ such that $\|u\|_p = l.$ Therefore, by using the Fatou's Lemma and $(u_k) \subset S_m$,
		we have that 
		$$0<l^p = \int_{\mathbb{R}^N} |u|^p dx \leq \liminf_{k \to \infty} \int_{\mathbb{R}^N}|u_k|^p dx = m^p.$$
		As a consequence, we obtain $l \in (0, m)$. Now, by using Brezis-Lieb Lemma, we deduce that
		\begin{equation}\label{bre}\begin{array}{ll}
				\|u_k\|^p_p = \|u_k - u\|^p_p + \|u\|_p^p + o_k(1)\\
				
				\|u_k\|^q_q = \|u_k - u\|^q_q + \|u\|_q^q + o_k(1).
			\end{array}
		\end{equation}
		Consider the sequence $v_k = u_k - u$. Define the auxiliary sequence $d_k = \|v_k\|^p_p$. Let us suppose that $\|v_k\|^p_p \to d^p$. Thus, by using \eqref{bre}, we conclude that $m^p = l^p + d^p$. Hence, $d < m$ proving that $d_k \in (0, m)$ is satisfied for each $k \in \mathbb{N}$ big enough. According to the Brezis-Lieb Lemma \cite{br}, we obtain that \begin{equation}\label{bre2}[u_k]^p = [u_k - u]^p + [u]^p + o_k(1).
		\end{equation}
		Under these conditions, by using \eqref{bre}, \eqref{bre2} together with Lemma \ref{important}, we infer that
		\begin{eqnarray}\nonumber \gamma_{m} + o_{k}(1) &=& J(u_k) = J(v_k) + J(u) + o_k(1) \geq \gamma_{d_k} + \gamma_{l} + o_{k}(1)\\
			\nonumber \allowdisplaybreaks &>& \frac{d^p_k}{m^p}\gamma_{m} + \gamma_{l} + o_k(1).\end{eqnarray}
		Now, doing $k \to \infty$, we mention that
		$\gamma_{m} \geq \frac{d^p}{m^p} \gamma_{m} + \gamma_{l}$. Since $l \in (0, m)$ we can once again apply the Lemma \ref{important} in order to ensure that $$\gamma_{m} > \frac{d^p}{m^p}\gamma_{m} + \frac{l^p}{m^p}\gamma_{m} = \left(\frac{d^p + l^p}{m^p}\right)\gamma_{m} = \gamma_{m}.$$ Thus, we obtain a contradiction. As a consequence, we deduce that $l = m$ proving that $u \in S_m$. This ends the proof. 
	\end{proof}
	
	\begin{lemma}\label{7} Assume $V \in L^{\infty}(\mathbb{R}^N), V \geq 0, V \not\equiv 0, \beta > 0$ $(V_3)$ and $q \in (p, p + sp^2/N)$. Let $(u_k) \subset S_m$ be a minimizing sequence for the Problem \eqref{minimization}. Suppose that $u_k \rightharpoonup u$ in $X$ for some $u \in X$ such that $u \not\equiv 0$. Then $u_k \to u$ in $X$.
	\end{lemma}
	\begin{proof}
		In view of Lemma \ref{6} we infer that $\|u_k\|_p = \|u\|_p = m$ and $u_k \rightharpoonup u$ in $L^p(\mathbb{R}^N)$. Hence, by using that $L^p(\mathbb{R}^N)$ is uniformly convex, we infer that $u_k \to u$ in $L^{p}(\mathbb{R}^N)$. Now, we claim that $u_k \to u$ in $L^q(\mathbb{R}^N)$. In fact, by using the interpolation law for the spaces $L^{t}(\mathbb{R}^N), t \in [p, p^*_s),$  we see that
		$\|u_k - u\|_q \leq \|u_k - u\|^t_p \|u_k - u\|^{1 - t}_{p^*_s}$, where $t = N \left(\frac{1}{sp} - \frac{1}{sq}\right)$. Recall that $\|u_k - u\|_p \to 0$ in $L^p(\mathbb{R}^N)$. It remains to prove that $\|u_k - u\|^{1 - t}_{p^*_s} \leq K $ holds for some $K > 0$. In order to do Proposition \ref{gagl}, we have
		$$\|u_k - u\|_{p^*_s} \leq K_1 [u_k - u] \leq K_2 \|u_k - u\| \leq K_3.$$ 
		It follows that from the last assertions that $\|u_k - u\|_q \to 0$ as $k \to \infty$. On the other hand, we observe that the norm is weakly lower semicontinuous. In particular, we deduce that
		$$ \|u\|^p = [u]^p + \int_{\mathbb{R}^N}V(x)|u|^p dx \leq \liminf_{k \to \infty} \left([u_k]^p + \int_{\mathbb{R}^N}V(x)|u_k|^p dx \right) .$$
		However, by using that $(u_k)$ is a minimizing sequence for the minimization Problem \eqref{minimization} together with the last estimates, we infer that
		$$ J(u) \le \liminf_{k \to \infty} J(u_k) = \gamma_{m} = \inf\{ J(w) : w \in S_m\}.$$
		Furthermore, by using the fact that $u \in S_m$,  we obtain that $J(u) = \gamma_{m}$. This completes the proof.
	\end{proof}
	
	\subsubsection{The proof of Theorem \ref{TB1} completed.}
	According to Lemma \ref{neg} we obtain that $\gamma_{m} < 0$. Now, applying Lemma \ref{nontrivial}, we obtain also the existence of a minimizing sequence $(u_k) \subset S_m$ such that $u_k \rightharpoonup u$ in $X$ where $u \not\equiv 0$. Furthermore, by using Lemmas \ref{6} and \ref{7}, we deduce that $u_k \to u$ in $X$, $J(u) = \gamma_{m}$ and $u \in S_m$.
	
	From now on, we consider $u \not\equiv 0$ a solution for the minimization Problem \eqref{minimization}, that is, we have that $\gamma_m = J(u) = \inf\{J(w): w \in S_m\}$. We assert that we can choose $u \geq 0$ in $\mathbb{R}^N$. In fact, by using the fact that $u \in S_m$, we obtain $|u| \in S_m$. Furthermore, $[|u|] \leq [u]$ is also verified, see \cite{jeff}. Therefore, we infer that $\gamma_m \leq J(|u|) \leq J(u) = \gamma_m$. As a consequence, using $|u|$ instead of $u$, we obtain a nonnegative minizer for the functional $J$ in the set $S_m$. Now, using Moser's iteration together with $L^r$ estimates, we obtain that $u \in L^{\infty} (\mathbb{R}^N) \cap L^{r}(\mathbb{R}^N)$ holds true for each $r \in [p, \infty)$, see for instance \cite{ambi}. Moreover, by using \cite[Corollary 5.5]{ia}, we infer that $ u 
 \in C_{loc}^{0, \alpha}(\mathbb{R}^N)$ holds for some $\alpha \in (0, 1)$. It follows from the strong maximum principle \cite[Theorem 1.2]{Maxfor} that $u = 0$ in $\mathbb{R}^N$ or $u > 0$ in $\mathbb{R}^N$. Since $u \in S_m$ we obtain that $u > 0$ in the whole $\mathbb{R}^N$. Ineed, arguing by contradiction, we prove that $u(x_0) = 0$ is impossible for each $x_0 \in \mathbb{R}^N$ fixed. Now, by using Lagrange Multiplier Theorem \cite{Dra}, we mention that
	$$J'(u) = \lambda \Psi'(u),\ \hbox{with},\ \Psi(u) = \frac{1}{p} \int_{\mathbb{R}^N} |u|^p dx.$$
	It easy to see that $$\left<u, \varphi\right> = \lambda \int_{\mathbb{R}^N} |u|^{p - 2}u\varphi dx + \beta\int_{\mathbb{R}^N} |u|^{q - 2}u \varphi dx,$$
	where $$\left<u, \varphi\right> = \int_{\mathbb{R}^N}\int_{\mathbb{R}^N} \frac{\left|u(x) - u(y) \right|^{p-2} \left(u(x) - u(y)\right)\left(\varphi(x) - \varphi(y)\right)}{\left|x - y\right|^{N + sp}} dx dy + \int_{\mathbb{R}^N}  V(x) \left|u\right|^{p-2} u\varphi dx, u, \varphi \in X.$$
	As a consequence, $u$ is a weak solution to the following nonlocal elliptic problem:
	\begin{equation}\nonumber \left\{\begin{array}{cc}
			\displaystyle (-\Delta)^s_p u\ +\ V (x) |u|^{p - 2}u\  = \lambda |u|^{p - 2}u + \beta\left|u\right|^{q-2}u\ \hbox{in}\ \mathbb{R}^N, \\
			\displaystyle \|u\|_p^p = m^p,\ u \in W^{s, p}(\mathbb{R}^N).\end{array}\right.
	\end{equation}
	This ends the proof.
	\hfill\cqd
	
	\section{$L^p$-subcritical Growth: The asymptotically periodic case}
	In this section we prove the existence of local minimizers to $J$ with prescribed norm and asymptotically periodic potential. The main idea here is to apply hypothesis $(V_2)$ in order to ensure existence of at least one weak solution for the Problem \eqref{P} with prescribed norm. For this purpose, we consider the  function $V_{\theta}: \mathbb{R}^N \to \mathbb{R}$ which  is $1$-periodic in such way that $V \not\equiv V_{\theta}$ where $V(x) \leq V_{\theta}(x), x \in \mathbb{R}^N$ and $(V - V_{\theta}) \in \mathcal{F}$.
	
	\begin{remark}\label{R1}  Let $V_{\theta} \in L^{\infty}(\mathbb{R}^N)$ a $1$-periodic potential where $\|V_{\theta}\|_{\infty} < \delta$. Consider the auxiliary functional
		\begin{equation}\label{perio}\displaystyle J_{\theta}(u) = \frac{1}{p} [u]^p + \frac{1}{p} \int_{\mathbb{R}^N} V_{\theta}(x) |u|^p dx - \frac{\beta}{q} \int_{\mathbb{R}^N} |u|^q dx \, \hbox{where}\, \gamma_{m, \theta} = \inf \{J_{\theta}(u): u \in S_m\}.
		\end{equation}
		Under these conditions, by using Theorem \ref{TB1}, there exists a function $u_{\theta} \in S_m$ with $u_{\theta} > 0$ such that $J_{\theta}(u_{\theta}) = \gamma_{m, \theta}$.
		Furthermore, by hypothesis $(V_2)$, there exists a measurable set $\Omega \subset \mathbb{R}^N$, $\mu\left(\Omega\right) > 0$ in such way that $V(x) < V_{\theta}(x), x \in \Omega$. Now, we claim that
  
$$\int_{\mathbb{\mathbb{R}^N}} \left[V(x) - V_{\theta}(x)\right]|u_\theta|^p dx < 0.$$ Otherwise, we obtain that 
$$\int_{\mathbb{\mathbb{R}^N}} \left[V(x) - V_{\theta}(x)\right]|u_\theta|^p dx = 0.$$
It follows that $\left[V(x) - V_{\theta}(x)\right]|u_\theta|^p = 0$ a. e. in $\mathbb{R}^N$. Since $u_\theta > 0$, by using $(V_2)$, we obtain a contradiction. Therefore, we obtain that
  $$\gamma_m \leq J(u_{\theta}) < J_{\theta}(u_{\theta}) = \gamma_{m, \theta},$$
Hence, $\gamma_{m} < \gamma_{m, \theta}$ holds true.
	\end{remark}
	
	In order to establish our main results we need to prove that for any minimizer sequence $(u_k) \subset S_m$ for the minimization Problem \eqref{minimization} such that $u_k \rightharpoonup 0$ satisfies the following condition 
	$$\displaystyle \lim_{k \to \infty} \int_{\mathbb{R}^N} \left|V(x) - V_\theta(x)\right||u_k|^p dx = 0.$$
	This result is quite important due to the fact that under assumption $(V_2)$ the energy functional $J$ is not invariant over translations. Moreover, hypothesis $(V_2)$ is more general that the standard one. Hence, we need to consider some auxiliary results and definitions.

	\begin{definition}\label{DR} Let us consider the following measurable sets:
		$$D_{\varepsilon} = \{x \in \mathbb{R}^N: |V(x) - V_{\theta}(x)| \geq \varepsilon\} \mbox{ and }   D_{\varepsilon}(R) = \{x \in \mathbb{R}^N: |V(x) - V_{\theta}(x)| \geq \varepsilon, |x| \geq R\}.$$
	\end{definition}

	\begin{proposition}\label{PB2}
		Suppose $(V_2)$. Then $\mu(D_{\varepsilon}(R)) \to 0$ as $R \to \infty$.
		
	\end{proposition}
	
	\begin{proof} The main idea here is to argue as was done in \cite{Elves}. Initially, we observe that $D_{\varepsilon}(R) = D_{\varepsilon} \cap B^c_R(0)$. Hence, the assertion $\mu(D_{\varepsilon}(R)) \to 0$ is equivalent to prove that $\mu(D_{\varepsilon} \cap B^c_{R}(0)) \to 0$ as $R \to \infty$. In order to do that we consider the function $h : \mathbb{R}^N \rightarrow \mathbb{R}$ given by 
		\begin{equation}h(x) = \left\{\begin{array}{ll} 1,\ \hbox{if}\ x \in D_{\varepsilon}\\
				0,\ \hbox{otherwise}.
			\end{array}\right.
		\end{equation}
		Therefore, by using the hypothesis $(V_2)$, we see that 
		$$\int_{\mathbb{R}^N} |h| dx = \int_{D_{\varepsilon}} |h| dx + \int_{D^c_{\varepsilon}} |h| dx = \int_{D_{\varepsilon}} |h| dx = \mu(D_{\varepsilon}) < \infty.$$
		The estimates just above ensure that  $h \in L^{1}(\mathbb{R}^N)$. Let us consider the sequence of functions $h_k: \mathbb{R}^N \rightarrow \mathbb{R}$ defined as follows:
		\begin{equation}
			h_k(x) = \left\{\begin{array}{ll} 1,\ \hbox{if}\ x \in D_{\varepsilon} \cap B^c_{R_k}(0)\\
				0,\ \hbox{otherwise}.
			\end{array}\right.
		\end{equation}
		Since $D_{\varepsilon} \cap B^c_{R_k}(0) \subset D_{\varepsilon}$ we conclude that $|h_k| < |h|$ for all $x \in \mathbb{R}^N$. Moreover, we observe that $h_k \to 0$ a. e. in $\mathbb{R}^N$. Therefore, applying the Dominated Convergence Theorem, we obtain that 
		$$\displaystyle 0 = \int_{\mathbb{R}^N} \lim_{k \to \infty} |h_k| dx = \lim_{k \to \infty} \int_{\mathbb{R}^N} |h_k| dx = \lim_{k \to \infty} \mu(D_{\varepsilon} \cap B^c_{R_k}(0)).$$
		The last assertion ensures that the desired result is satisfied. This ends the proof. 
	\end{proof}
	
	Now, by using Proposition \ref{PB2}, we can prove the following powerful result:
	\begin{proposition}\label{PB4} Assume that $(V_2)$ and $(V_3)$ hold. Let $(u_k) \subset X$ be a sequence such that $u_k \rightharpoonup 0$ in $X$. Then
		\begin{equation}\label{eqB10}\displaystyle \lim_{k \to \infty} \int_{\mathbb{R}^N} |V(x) - V_{\theta}(x)||u_k|^p dx = 0.\end{equation}
	\end{proposition}

	\begin{proof} 
		Firstly, we consider the set $D_{\varepsilon}(R)$ given in Definition \ref{DR}. 
		Hence, by using Proposition \ref{PB2}, we know that there exists a value $R = R(\varepsilon)$ large enough in such way that  $\mu(D_{\varepsilon}(R)) < \delta$ where $\delta>0$ is arbitrary and $0 < \varepsilon < \delta$. In particular, we obtain that
		\begin{equation}\label{e2}
			\int_{\mathbb{R}^N} |V(x) - V_{\theta}(x)| |u_k|^p dx = \int_{B_R(0)} |V(x) - V_{\theta}(x)| |u_k|^p  dx + \int_{B^c_R(0)} |V(x) - V_{\theta}(x)| |u_k|^p  dx. 
		\end{equation}
		In order to prove \eqref{eqB10} we shall consider the following items:
		\begin{itemize}
			\item[i)] $\displaystyle \int_{B_R(0)} |V(x) - V_{\theta}(x)| |u_k|^p  dx = o_k(1);$ 
			\item[ii)] $\displaystyle \int_{B^c_R(0)} |V(x) - V_{\theta}(x)| |u_k|^p  dx = o_k(1).$
		\end{itemize}	
		Firstly, we shall prove that item $i)$ is verified. Notice that $u_k \to 0$ a. e. in $\mathbb{R}^N$. Hence, $|V(x) - V_{\theta}(x)||u_k|^p   \to 0$ as $k \to \infty$. Furthermore, by using the fact that $u_k \to 0$ in $L^p_{loc}(\mathbb{R}^N)$, we deduce that 
		$$|V(x) - V_{\theta}(x)| |u_k|^p \leq \|V - V_{\theta}\|_{\infty} h_p^p \in L^1(B_R(0)).$$
		Therefore, by using the Dominate Convergence Theorem, we infer that
		$$\displaystyle \lim_{k \to \infty} \int_{B_R(0)} |V(x) - V_{\theta}(x)| |u_k|^p  dx = 0.$$
		
		At this stage, we shall prove the item $ii)$. Notice also that $(u_k)$ is bounded in $X$. Therefore, by the using Hölder inequality with $r/p$ and $r/(r-p)$ where $r \in (p, p^*_s)$, we deduce that
		\begin{eqnarray}\nonumber
			\displaystyle \int_{B^c_R(0)\cap D_{\varepsilon}} |V(x) - V_{\theta}(x)| |u_k|^p dx &=& \int_{D_{\varepsilon}(R)} |V(x) - V_{\theta}(x)| |u_k|^p dx \\
			\nonumber \displaystyle &\le& \left(\int_{D_{\varepsilon}(R)} |V(x) - V_{\theta}(x)|^{\frac{r}{r -p}}dx\right)^{\frac{r - p}{r}} \left(\int_{D_{\varepsilon}(R)} |u_k|^{r} dx\right)^{\frac{p}{r}}\\
			\label{e3} \displaystyle &\leq& \|V - V_{\theta}\|_{\infty} [\mu(D_{\varepsilon}(R))]^{\frac{r - p}{r}} S_r^p \|u_k\|^p < K\delta^{\frac{r - p}{r}}. \end{eqnarray}
		Furthermore, by using again that $(u_k)$ bounded in $X$, we also infer that
		\begin{equation}\label{e4}
			\int_{B_R^c(0) \cap D_{\varepsilon}^c} |V(x) - V_{\theta}| |u_k|^p dx \le \int_{D_{\varepsilon}^c} |V(x) - V_{\theta}| |u_k|^p dx \le \varepsilon \int_{\mathbb{R}^N} |u_k|^p dx \le \varepsilon K_1 < \delta K_1. \end{equation}
		Now, by using the estimates \eqref{e3} and \eqref{e4} and taking into account \eqref{e2}, we mention that
		$$\int_{\mathbb{R}^N} |V(x) -V_{\theta}| |u_k|^p dx < K \delta^{\frac{r - p}{r}} + K_1 \delta.$$
		Since $\delta>0$ is arbitrary the proof for desired result follows.\end{proof}

	\begin{lemma}\label{L4} Assume that $V$ satisfies $(V_2)$ and $(V_3)$  and $\beta > 0$. Let $(u_k) \subset S_m$ be a minimizing sequence for the Problem \eqref{minimization}. Then $u_k \rightharpoonup u$ in $X$ and $u \not\equiv 0$.
	\end{lemma}	\begin{proof}
		Let $(u_k) \subset S_m$ be a minimizing sequence for the minimization Problem \eqref{minimization}. In this way, using the Proposition \ref{lim} we obtain that $(u_k)$ is bounded. As a consequence, up to a subsequence, there exists $u \in X$ such that $u_k \rightharpoonup u$ in $X$. The proof follows arguing by contradiction. Assume that $u \equiv 0$. Then, we obtain that
		\begin{eqnarray}\nonumber \displaystyle \gamma_{m} + o_k(1) &=& J(u_k) = \dfrac{1}{p}[u_k]^p + \dfrac{1}{p}\int_{\mathbb{R}^N}V(x)|u_k|^p dx - \dfrac{\beta}{q}\int_{\mathbb{R}^N} |u_k|^q dx \\
			\nonumber \displaystyle &=&J_{\theta}(u_k) + \frac{1}{p} \int_{\mathbb{R}^N} (V(x) - V_{\theta}(x))|u_k|^p dx. \end{eqnarray}
		Therefore, by using the definition of $\gamma_{m, \theta}$, see \eqref{perio}, we obtain that
		\begin{equation}\label{e1}
			\gamma_{m} + o_k(1) = J(u_k) \geq \gamma_{m, \theta} + \frac{1}{p} \int_{\mathbb{R}^N} (V(x) - V_{\theta}(x))|u_k|^p dx.
		\end{equation}
		It follows from Proposition \ref{PB4} that 
		$$\displaystyle\lim_{k \to \infty}\int_{\mathbb{R}^N} (V(x) - V_{\theta}(x))|u_k|^p dx = 0.$$
		Therefore, doing $k \to \infty$ in \eqref{e1}, we deduce that $\gamma_{m} \geq \gamma_{m, \theta}$. However, by using Remark \ref{R1} we arrive at a contradiction due to the fact that $\gamma_{m} < \gamma_{m, \theta}$. This finishes  the proof.
	\end{proof}

	\noindent\textit{The proof of Theorem \ref{TB2}.} Initially, by using Lemma \ref{neg}, we observe that $\gamma_{m, \theta} < 0$.  Therefore, by using Remark \ref{R1}, we obtain that $\gamma_{m} < 0$. According to Lemma \ref{L4} we obtain that there exists a minimizing sequence $(u_k) \subset S_m$ such that $u_k \rightharpoonup u$ in $X$ with $u \not\equiv 0$. In view of Lemmas \ref{6} and \ref{7} we deduce also that $u_k \to u$ in $X$ and $J(u) =\gamma_{m}$. As a consequence, by using the same ideas employed was done in the proof of Theorem \ref{TB1}, $u > 0$ in $\mathbb{R}^N$, and there exists $\lambda$ such that $u$ is a weak solution for the Problem \eqref{P}. 	\hfill\cqd
	
	\section{$L^p$-subcritical growth: Existence of weak solution for Problem \eqref{gammameps}}
	In this section, we remember that $J_\varepsilon: X \to \mathbb{R}$, is given by  $$J_\varepsilon(u) = \frac{1}{p}[u]^p + \frac{1}{p} \int_{\mathbb{R}^N} V(\varepsilon x)|u|^p dx - \frac{\beta}{q} \int_{\mathbb{R}^N} |u|^q dx.$$
	Our main objective here is to show that the functional $J_{\varepsilon}$ has at least one local minimizer with prescribed norm in the case $L^p$ subcritical. In particular, by using the Lagrange Multiplier Theorem \cite{Dra}, we shall prove that the local minimizer is a weak solution to the following nonlocal elliptic problem:
	\begin{equation}\label{Equi}\tag{$P_{m, \varepsilon}$}
		\left\{\begin{array}{cc}
			\displaystyle (-\Delta)^s_p u\ +\ V (\varepsilon x) |u|^{p-2}u\  = \lambda |u|^{p - 2}u + \beta\left|u\right|^{q-2}u\ \hbox{in}\ \mathbb{R}^N, \\
			\displaystyle \|u\|_p^p = m^p,\ u \in W^{s, p}(\mathbb{R}^N),\end{array}\right.
	\end{equation}
	where $\varepsilon > 0$ is small enough. In order to prove the Theorem \ref{TB3}  we consider some auxiliary functionals associated with $V_0$ and $V_{\infty}$ defined in $(V_4)$. More precisely, we consider $J_{0}:X \rightarrow \mathbb{R}$ and $J_{\infty}: X \rightarrow \mathbb{R}$ given by 
	\begin{eqnarray}J_0(u) &=& \frac{1}{p} [u]^p  + \frac{1}{p} \int_{\mathbb{R}^N} V_0 |u|^p dx - \frac{\beta}{q} \int_{\mathbb{R}^N} |u|^q dx\\
		\nonumber \allowdisplaybreaks J_{\infty}(u) &=& \frac{1}{p} [u]^p  + \frac{1}{p} \int_{\mathbb{R}^N} V_{\infty} |u|^p dx - \frac{\beta}{q} \int_{\mathbb{R}^N} |u|^q dx.
	\end{eqnarray}
	
	Now, using the functionals given just above, we consider also the following minimization problems
	\begin{eqnarray}\label{Pm0} \displaystyle \gamma_{m, 0} &=& \inf\{J_0(u) : u \in S_m\}\\
		\label{Pinf}\displaystyle \gamma_{m, \infty} &=& \inf\{J_{\infty}(u) : u \in S_m\} .\end{eqnarray} 
	It is important to stress that hypothesis $(V_4)$ together with $\|V\|_{\infty} < \delta$ imply that $V_0 < V_{\infty} < \delta$. Furthermore, by using Theorem \ref{TB1}, there exist  $u_0 \in S_m$ and $u_{\infty} \in S_m$ such that $J_0(u_0) = \gamma_{m, 0}$ and $J_{\infty} (u_{\infty}) = \gamma_{m, \infty}$.

	\begin{proposition}\label{R2} Assume that $V$ satisfies the condition $(V_4)$ and $\beta > 0$. Then $\gamma_{m, 0} < \gamma_{m, \infty} < 0$. 
	\end{proposition}
	\begin{proof}Let $u_0, u_{\infty} \in S_m$ be two functions given by Theorem \ref{TB1} in such way that $\gamma_{m, 0} = J_{0}(u_0)$ and $\gamma_{m, \infty} = J_{\infty}(u_{\infty})$. In particular, we obtain that
		\begin{eqnarray}\nonumber \displaystyle \gamma_{m, 0} &=& J_{0}(u_0) = \dfrac{1}{p}[u_0]^p + \dfrac{V_0 m^p}{p} - \dfrac{\beta}{q}\int_{\mathbb{R}^N}|u_0|^q dx\\
			\nonumber \displaystyle \gamma_{m, \infty} &=& J_{\infty}(u_\infty) = \dfrac{1}{p}[u_\infty]^p + \dfrac{V_\infty m^p}{p} - \dfrac{\beta}{q}\int_{\mathbb{R}^N}|u_\infty|^q dx.
		\end{eqnarray}
		In order to show that $\gamma_{m, 0} < \gamma_{m, \infty}$ is sufficient to ensure that  $$\dfrac{1}{p}[u_0]^p - \dfrac{\beta}{q}\int_{\mathbb{R}^N}|u_0|^q dx \leq \dfrac{1}{p}[u_\infty]^p - \dfrac{\beta}{q}\int_{\mathbb{R}^N}|u_\infty|^q dx.$$
		The proof for the last assertion follows arguing by contradiction. Let us assume that $$\dfrac{1}{p}[u_0]^p - \dfrac{\beta}{q}\int_{\mathbb{R}^N}|u_0|^q dx > \dfrac{1}{p}[u_\infty]^p - \dfrac{\beta}{q}\int_{\mathbb{R}^N}|u_\infty|^q dx.$$
		Hence, using the term $V_0 m^p/p$ in both sides of the inequality just above, we can see that $J_{0}(u_0) > J_{0}(u_{\infty})$. On the other hand, we observe that $u_{\infty} \in S_m$. Under these conditions and taking into account the fact that $\displaystyle \gamma_{m, 0} = \inf\{J_0(u): u \in S_m\} =J_{0}(u_0)$ we obtain a contradiction. This finishes the proof.
	\end{proof}

	\begin{lemma}\label{L6}Assume that $V$ satisfies $(V_3),$  $(V_4)$ and $\beta > 0$. Then, there exists $\varepsilon_0 > 0$ such that $\gamma_{m, \varepsilon} < \gamma_{m, \infty}$ for all $\varepsilon \in (0, \varepsilon_0)$.
	\end{lemma}
	\begin{proof} Let $x_0 \in \mathbb{R}^N$ be a fixed point in such way that $\displaystyle V(x_0) = \inf_{x \in \mathbb{R}^N} V(x)$. Define the function $v_{\varepsilon}(x) = u_0(x - \frac{x_0}{\varepsilon})$. Now, by using the change of variables $y = x - \frac{x_0}{\varepsilon}$, we obtain that
		$$\int_{\mathbb{R}^N} |v_\varepsilon(x)|^p dx = \int_{\mathbb{R}^N} |u_0(y)|^p dy = m^p.$$
		The last assertion implies that $v_{\varepsilon} \in S_m$. As a consequence, we see that $\gamma_{m, \varepsilon} \le J_{\varepsilon}(v_{\varepsilon}(x))$. Therefore, we obtain that 
		$$\displaystyle \gamma_{m, \varepsilon} \le \frac{1}{p}\int_{\mathbb{R}^N}\int_{\mathbb{R}^N} \frac{|v_\varepsilon (x) - v_\varepsilon (y)|^p}{|x - y|^{N + sp}} dx dy + \frac{1}{p}\int_{\mathbb{R}^N} V(\varepsilon x) |v_\varepsilon (x)|^p dx - \frac{\beta}{q} \int_{\mathbb{R}^N} |v_\varepsilon (x)|^q dx.$$
		Once again we consider the change of variables $w = x - \frac{x_0}{\varepsilon}$ and $z = y - \frac{x_0}{\varepsilon}$. It is not hard to see that
		$$\gamma_{m, \varepsilon} \le \frac{1}{p}[u_0]^p + \frac{1}{p}\int_{\mathbb{R}^N} V(\varepsilon w + x_0) |u_0|^p dw - \frac{\beta}{q} \int_{\mathbb{R}^N} |u_0|^q dw.$$
		Now, by using the Dominate Convergence Theorem together with $V_0 = V(x_0)$, we deduce that
		$$\gamma_{m, \varepsilon} \le \frac{1}{p}[u_0]^p + \lim_{\varepsilon \to 0} \frac{1}{p}\int_{\mathbb{R}^N} V(\varepsilon w + x_0) |u_0|^p dw - \frac{\beta}{q} \int_{\mathbb{R}^N} |u_0|^q dw = J_{0}(u_0) = \gamma_{m, 0}.$$
		In particular, we see that $ \gamma_{m, \varepsilon} \le \gamma_{m, 0}$. Now, by using the Proposition \ref{R2}, there exists $\varepsilon_0 > 0$ such that $\gamma_{m, \varepsilon} < \gamma_{m, \infty}$ for all $\varepsilon \in (0, \varepsilon_0)$. This ends the proof.
	\end{proof}
	At this stage, we consider a minimizing sequence $(u_k) \subset S_m$ for the minimization Problem \eqref{gammameps}, that is, we have $\displaystyle \lim_{k \to \infty} J_\varepsilon(u_k) = \gamma_{m, \varepsilon}$. Under these conditions, by using a similar argument as was done in the proof of  Proposition \ref{lim}, we obtain that $(u_k)$ is bounded in $X$. Thus, up to a subsequence, there exists a function $u \in X$ such that $u_k \rightharpoonup u$ in $X$ and $u_k \to u$ a. e. in $\mathbb{R}^N$. As a product, we can prove the following result:
	
	\begin{lemma}\label{L7} Suppose that $V$ satisfies $(V_3)$ and $(V_4)$, $\beta > 0$. Assume also that $\varepsilon \in (0, \varepsilon_0)$ where $\varepsilon_0$ is small enough. Let $(u_k)$ be a minimizing sequence for the minimization Problem \eqref{gammameps} in such way that $u_k \rightharpoonup u$ in $X$. Then $u \not\equiv 0$.    
	\end{lemma}
	
	\begin{proof}
		Once again the proof follows arguing by contradiction. Let us assume that $u_k \rightharpoonup 0$ in $X$. Hence, we obtain that 
		$$\gamma_{m, \varepsilon} + o_k(1) = J_{\varepsilon}(u_k) = \frac{1}{p}[u_k]^p + \frac{1}{p} \int_{\mathbb{R}^N} V(\varepsilon x) |u_k|^p dx - \frac{\beta}{q} \int_{\mathbb{R}^N} |u_k|^q dx.$$
		As a product, we obtain that
		\begin{equation}\nonumber  \gamma_{m, \varepsilon} + o_k(1) = J_{\infty}(u_k) + \frac{1}{p} \int_{\mathbb{R}^N} ( V(\varepsilon x) - V_{\infty})|u_k|^p dx.
		\end{equation}
		Now, by using the hypothesis $(V_4)$, we obtain that for all $ \zeta > 0$ there exists a real value $R > 0$ such that $V(x) \geq V_{\infty} - \zeta$, $x \in B^c_{R}(0)$. As a consequence, we obtain that
		\begin{eqnarray}\nonumber \displaystyle \gamma_{m, \varepsilon} + o_k(1) = J_{\varepsilon}(u_k) &=& J_{\infty}(u_k) + \int_{\mathbb{R}^N} (V(\varepsilon x) - V_{\infty})|u_k|^p dx\\
			\nonumber
			\displaystyle  &=& J_{\infty}(u_k) + \int_{B_{\frac{R}{\varepsilon}}(0)} (V(\varepsilon x) - V_{\infty})|u_k|^p dx\
			+ \int_{B^c_{\frac{R}{\varepsilon}}(0)} (V(\varepsilon x) - V_{\infty})|u_k|^p dx \\
			\label{e7} \displaystyle  &\geq& J_{\infty} (u_k) + \int_{B_{\frac{R}{\varepsilon}}(0)} (V(\varepsilon x) - V_{\infty})|u_k|^p dx - \zeta \int_{B^c_{\frac{R}{\varepsilon}}(0)} |u_k|^p dx
		\end{eqnarray}
		Recall also that the sequence $(u_k)$ is bounded in $X$. Indeed, $(u_k)$ is minimizing sequence for the Problem \eqref{gammameps} where $u_k \rightharpoonup u$ in $X$. It is important observe that $X$ is compactly embedded into $L^p(B_{\frac{R}{\varepsilon}}(0))$. Here, we emphasize that $\varepsilon$ is a fixed parameter. Furthermore, $u_k \to u$ in $L^p(B_{\frac{R}{\varepsilon}}(0))$. Therefore, by using the Dominated Convergence Theorem, we see that
		\begin{equation}\label{e8}\displaystyle \lim_{k \to \infty} \int_{B_{\frac{R}{\varepsilon}}(0)} (V(\varepsilon x) - V_{\infty})|u_k|^p dx = 0. \end{equation}
		On the other hand, we observe that 
		\begin{equation}\label{e9}
			-\zeta \int_{B^c_{\frac{R}{\varepsilon}}(0)} |u_k|^p dx \geq -\zeta C.
		\end{equation}
		
Now, taking into account \eqref{e7}, \eqref{e8} and \eqref{e9}, we conclude that
		$$\gamma_{m, \varepsilon} + o_k(1) \geq J_{\infty} (u_k) - \zeta C \geq \gamma_{m, \infty} - \zeta C.$$
		Furthermore, by using the fact that $\zeta$ is arbitrary, we obtain $\gamma_{m, \varepsilon} \geq \gamma_{m, \infty}$. The last inequality is a contradiction with the Lemma \ref{L6}. Hence, $u \not\equiv 0$ is now satisfied. This ends the proof.
	\end{proof}
	
	\noindent\textit{Proof of Theorem \ref{TB3}}. Let $(u_k) \subset S_m$ be a minimizing sequence for the minimization Problem \eqref{gammameps}. Hence, by using Proposition \ref{lim} and Lemma \ref{L7}, we have that $(u_k)$ is bounded and $u_k \rightharpoonup u$ in $X$ for some $u \not\equiv 0, u \in X$. Therefore, by using Lemma \ref{6}, we concluded that $u_k \to u$ in $X$ and $J_\varepsilon(u) = \gamma_{m, \varepsilon}$. Consequently, $u \in S_m$ is a local minimizer for the functional $J_{\varepsilon}$ restricted to $S_m$. Using similar ideas discussed in the proof of Theorem \ref{TB1}, $u > 0$ and there exists $\lambda$ such that
	$$J'_{\varepsilon}(u)\varphi = \lambda \Psi'(u)\varphi, \forall \varphi \in X.$$
	Therefore, $u$ is a weak solution for the Problem \eqref{Equi}. This completes the proof. \hfill\cqd
	
	\section{ $L^p$-critical growth: existence of weak solution }
	In this section we shall prove the existence of a solution for the following minimization problem:
	\begin{equation}\label{minimization3}
		\gamma_m = \inf\{J(u): u \in S_m\}
	\end{equation} 
	where $q = p + sp^2/N.$
	Firstly, we shall ensure that the solution to the minimization Problem \eqref{minimization3} is a weak solution to the nonlocal elliptic Problem \eqref{P}.
	Recall that $$J(u) = \dfrac{1}{p}[u]^p + \dfrac{1}{p}\int_{\mathbb{R}^N} V(x)|u|^p dx - \dfrac{\beta}{q}\int_{\mathbb{R}^N}|u|^q dx, u \in X.$$ 
	
	Now, by using the same ideas discussed in Section \ref{SB1}, we consider the fibering map $u_t(x) = t^{N/p}u(tx), t > 0, x \in \mathbb{R}^N$. It is easy to see that
	$$J(u_t) = \dfrac{t^{sp}}{p}[u]^p + \dfrac{1}{p}\int_{\mathbb{R}^N}V\left(\frac{x}{t}\right)|u|^p dx - \dfrac{\beta t^{\frac{N(q-p)}{p}}}{q}\int_{\mathbb{R}^N}|u|^q dx.$$
	
	\begin{remark}\label{igual} In the $L^p$ critical case we have that $q = p + sp^2/N$. In particular, the last identity implies that $sp = (q - p)N/p$.    
	\end{remark}
	Now, we shall prove that for each $\beta > 0$ small enough, the energy functional $J$ is bounded from below. More specifically, we obtain the following result.
	\begin{lemma}\label{limit} Suppose that $V$ satisfies $(V_3)$ , $(V_5)$ and $V \geq 0$. Then, there exists $\beta_0>0$ such that $J(u)>0$ for every $\beta\in (0,\beta_0)$ and $u\in S_m$. In particular, we obtain that $ \gamma_m \geq 0$.   
	\end{lemma}
	\begin{proof} Let $u \in S_m$ be a fixed function. Now, by using interpolation law for the spaces $L^q(\mathbb{R}^N)$ together with the Gagliardo-Nirenberg inequality, we obtain that
		\begin{eqnarray}\nonumber \displaystyle J(u) &\geq& \dfrac{1}{p}[u]^p + \dfrac{1}{p}\int_{\mathbb{R}^N} V(x)|u|^p dx - \dfrac{\beta C m^{q-p}}{q}[u]^p \\
			&\geq& \dfrac{1}{p}[u]^p + \dfrac{1}{p}\int_{\mathbb{R}^N} V(x)|u|^p dx - \dfrac{\beta C m^{q-p}}{q}\|u\|^p \nonumber \\
			\nonumber \displaystyle &=& \left(\dfrac{1}{p} - \dfrac{\beta C m^{q-p}}{q}\right)\|u\|^p > 0.
		\end{eqnarray}
		Here we mention that the last inequality is satisfied for each  $$\beta \in \left(0, \dfrac{q}{pCm^{q - p}}\right).$$ Since $q = p+sp^2/N$ we consider $$\beta_0 = \dfrac{N + sp}{pCm^{\frac{sp^2}{N}}}.$$ This completes the proof.
	\end{proof}

	\begin{remark}\label{ob1} In view of Lemma \ref{limit} we obtain  that 
		$$J(u) \geq \left(\dfrac{1}{p} - \frac{\beta C m^{q-p}}{q}\right)\|u\|^p > 0$$ holds for each $\beta \in (0, \beta_0)$. Thus, given a minimizing sequence $(u_k) \subset S_m$ for the minimization Problem \eqref{minimization3}, that is, $\displaystyle \lim_{k \to \infty} J(u_k) = \gamma_m$, we obtain that $(u_k)$ is bounded in $X$.   
	\end{remark}
	
	In light of hypothesis $(V_5)$ we are able to establish the following result:
	
	\begin{proposition}\label{compa}
		Suppose that $(V_3)$ and $(V_5)$ are satisfied. Then the embedding $X\hookrightarrow L^{t}(\mathbb{R}^N)$, $t \in [p, p^*_s)$ is compact.
	\end{proposition}
	
	\begin{proof}
		Firstly, we observe that the embedding $X\hookrightarrow L^{t}(\mathbb{R}^N)$, $t \in [p, p^*_s)$ is continuous due to the fact that $(V_3)$ is verified. In fact, by using $(V_3)$, we deduce that
		\begin{equation}
			\|u\|_p^p \leq \dfrac{1}{\sigma} \|u\|^p, u \in X.
		\end{equation}
		Let $(u_k) \subset X$ be a sequence satisfying $u_k \rightharpoonup u$ in $X$ for some $u \in X$. In particular, we know that $(u_k)$ is a bounded sequence in $X$. Here is sufficient to ensure that $||u_k - u||_t \to 0$ for all $t \in [p, p^*_s)$. Initially, we shall consider the case $t = p$. Recall that 
		\begin{equation}\label{comp}
			\|u_k - u\|_p^p = \int_{B^c_R(0)} |u_k - u|^p dx + \int_{B_R(0)}|u_k - u|^p dx.
		\end{equation}		
		Now, due to the compact embedding $X\hookrightarrow L^{p}(B_R(0))$, see for instance \cite{neza}, we see that $$\displaystyle \lim_{k \to \infty} \int_{B_R(0)}|u_k - u|^p dx = 0.$$
		Furthermore, we claim that $$\displaystyle \lim_{k \to \infty}\int_{B^c_{R}(0)} |u_k - u|^p dx = 0.$$ 
		In order to prove the desired claim we consider some sets as follows:
		$$B^c_R(0) = \left[B^c_R(0) \cap \mathcal{A}_V\right] \cup \left[B^c_R(0) \cap \mathcal{A}_V^c\right],$$ where $$\mathcal{A}_V = \{x \in \mathbb{R}^N : V(x) \geq M\},\ \ \ \ \ \mbox{ for every } M>0.$$ 
		As a consequence, we obtain
		$$\displaystyle \int_{B^c_R(0)} |u_k - u|^p dx = \int_{B_R^c(0) \cap \mathcal{A}_V} |u_k - u|^p dx + \int_{B^c_R(0) \cap \mathcal{A}_V^c} |u_k - u|^p dx.$$
		Now, due to the fact that $M > 0$ is arbitrary and $u_k \rightharpoonup u$ in $X$, we infer that 
		\begin{equation}\label{c1}\displaystyle \int_{B^c_R(0) \cap \mathcal{A}_V} |u_k - u|^p dx \le \dfrac{1}{M}\int_{B^c_R(0) \cap \mathcal{A}_V} V(x)|u_k - u|^p dx \le\dfrac{1}{M}\|u_k - u\|^p < \dfrac{\varepsilon}{2}. \end{equation}
		On the other hand, using the Hölder inequality with exponents $r = (p^*_s)/(p^*_s - p)$ and $p^*_s/p$, we deduce that
		$$ \int_{B^c_R(0) \cap \mathcal{A}_V^c} |u_k - u|^p dx \leq \left[\mu(B^c_R(0) \cap \mathcal{A}_V^c)\right]^{\frac{1}{r}} C \|u_k - u\|_{p^*_s}^p \leq C_1 \left[\mu(B^c_R(0) \cap \mathcal{A}_V^c)\right]^{\frac{1}{r}} .$$
		Furthermore, we observe that  $\displaystyle \lim_{R \to \infty }\mu(B^c_R(0) \cap \mathcal{A}_V^c) = 0 $. Indeed, we infer that
		$$\displaystyle \lim_{R \to \infty}\mu(B^c_R(0) \cap \mathcal{A}_V^c) = \lim_{R \to \infty}\int_{B^c_R(0) \cap \mathcal{A}_V^c} dx \le \lim_{R \to \infty}\int_{\mathcal{A}_V^c} \mathcal{X}_{B^c_R(0) \cap \mathcal{A}^c_V} dx = 0. $$
		Here was used the Dominated Convergence Theorem together with the fact that $\mu(B_R^c(0) \cap \mathcal{A}_V^c) \leq \mu(\mathcal{A}_V^c) < \infty$. As a consequence, for every $\varepsilon > 0$, there exists $R = R(\varepsilon)$ in such way that
		\begin{equation}\label{c2} 
			\int_{B^c_R(0) \cap \mathcal{A}_V^c} |u_k - u|^pdx < \frac{\varepsilon}{2}.     
		\end{equation}
		In view of \eqref{c1} and \eqref{c2}, we obtain that
		$$\int_{B_R^c(0)} |u_k - u|^p dx = \int_{B_R^c(0) \cap \mathcal{A}_V} |u_k - u|^p dx + \int_{B^c_R(0) \cap \mathcal{A}_V^c} |u_k - u|^p dx < \varepsilon.$$
		It follows from \eqref{comp} that $$\lim_{k \to \infty}\|u_k - u\|_p^p = 0.$$ Now, by using the interpolation law for the spaces $L^t(\mathbb{R}^N), t \in [p, p^*_s)$, we conclude the desired result.
	\end{proof}

	\begin{lemma}\label{weak} Suppose $(V_3)$ and $(V_5)$. Let $(u_k) \subset S_m$ be a minimizing sequence for the minimization Problem \eqref{minimization3}. Then, up to a subsequence, $u_k \rightharpoonup u$ in $X$ where $u \in S_m$.
	\end{lemma}
	\begin{proof}
		Let $(u_k) \subset S_m$ be a minimizing sequence for the minimization Problem \eqref{minimization3}, that is,  $\displaystyle \gamma_m = \lim_{k \to \infty}J(u_k)$. Now, by using Remark $\ref{ob1}$, we see that $(u_k)$ is bounded in $X$. Hence, up to a subsequence, there exists a function $u \in X$ such that $u_k \rightharpoonup u$ in $X$. It remains to prove that $u \not\equiv 0$. In order to do that we apply the compact embedding given in Proposition \ref{compa} proving that there exists $h_t \in L^{t}(\mathbb{R}^N)$ in such way that
		\begin{equation}\left\{\begin{array}{lll}u_k \to u \in L^t(\mathbb{R}^N), t \in [p, p^*_s)\\
				u_k \to u\ \hbox{a. e. in}\ \mathbb{R}^N\\
				|u_k| \leq h_t, h_t \in L^t(\mathbb{R}^N), t \in [p, p^*_s)
			\end{array}\right.
		\end{equation}
		In particular, taking $t = p$, we can apply the Dominated Convergence Theorem proving that
		$$\lim_{k \to \infty} \int_{\mathbb{R}^N} |u_k|^p dx = \int_{\mathbb{R}^N}|u|^p dx = m^p > 0.$$
		As a consequence, we obtain that   $u \not\equiv 0$ and $u \in S_m$. This ends the proof.
	\end{proof}
	
	\textit{The proof of Theorem \ref{TB4}} Let $(u_k)$ be a minimizing sequence for the minimization Problem \eqref{minimization3}. Then, up to a subsequence, there exists  $u \in S_m$ such that $u_k \rightharpoonup u$ in $X$, see Lemma \ref{weak}. Therefore, by using the compact embedding given in Proposition \ref{compa} and taking into account that the norm is lower weakly semicontinuous, we infer that
	\begin{eqnarray}\nonumber \displaystyle J(u) &=& \dfrac{1}{p}[u]^p + \dfrac{1}{p}\int_{\mathbb{R}^N}V(x)|u|^p dx - \dfrac{\beta}{q}\int_{\mathbb{R}^N} |u|^q dx\\
		\nonumber \displaystyle &\leq& \liminf_{k \to \infty}\frac{1}{p}\|u_k\|^p - \lim_{k \to \infty} \dfrac{\beta}{q}\|u_k\|^q_q = \lim_{k \to \infty} J(u_k) = \gamma_m.
	\end{eqnarray}
	Hence, $J(u) = \gamma_m$, that is, $u$ is a local minimizer for $J$ restricted to $S_m$. Furthermore, arguing as was done in the proof of Theorem \ref{TB1}, we infer that $u > 0$ in $\mathbb{R}^N$. Hence, there exists $\lambda$ in such way that $J'(u)\varphi = \lambda\Psi'(u)\varphi$ holds for each $\varphi \in X$. Here, we emphasize that $\Psi: X \to \mathbb{R}$ given by $$\Psi (u) = \frac{1}{p}\int_{\mathbb{R}^N}|u|^p dx, u \in X$$ is a $C^1$ functional. Therefore, the function $u$ satisfies the following identity
	$$\left<u, \varphi\right> = \lambda \int_{\mathbb{R}^N} |u|^{p - 2}u\varphi dx + \beta\int_{\mathbb{R}^N} |u|^{q - 2}u \varphi dx. $$ As a consequence, $u$ is a weak solution to Problem \eqref{P}. 
	This ends the proof. \hfill\cqd
	
	\subsection*{Acknowledgments}
	The  author first was also partially supported by CNPq with grant 309026/2020-2. The second author was partially supported by Fapeg with grant 202110267000424.
	
	\subsection*{Author Contribution}	In the present work authors have the same contribution. All authors also reviwed the manuscript in the present version for this work.
	
	\subsection*{Funding} Edcarlos D. Silva thanks to CNPq for financial support through the Project 309026/2020-2, Brazil.
	J. L. A. Oliveira thanks to Fapeg for financial support through the Project 202110267000424, Brazil.
	
	\subsection*{Availability of data and material} For the present work the availability of data and material is not applicable.
	
	\subsection*{Ethics approval and consent to participate} This item is not applicable for the present work.
	
\subsection*{Consent for publication} All authors read and approved the final version for the present manuscript.
	
	\subsection*{Declarations}
	
	\subsection*{Conflict of Interest}
	The authors declare that they have no conflict of interest.
	
	

\end{document}